\documentclass[onefignum,onetabnum]{siamart220329}

\usepackage{latexsym}
\usepackage{amssymb,amsbsy,amsmath,amsfonts,amssymb,amscd}
\usepackage{mathrsfs,hyperref}
\usepackage{epsfig}
\graphicspath{{./figures/}}
\usepackage[margin = 1 in]{geometry}

\usepackage{amsfonts}
\usepackage{graphicx}
\usepackage[caption=false]{subfig}
\usepackage{amsmath}
\usepackage{amssymb}
\usepackage{amsmath}
\usepackage{algorithm}
\usepackage{algorithmicx}
\usepackage{algpseudocode}
\usepackage{bm}
\usepackage{dsfont}
\usepackage{mathtools}
\usepackage{comment}
\usepackage{cases}
\usepackage{cleveref}
\newcommand{\rd}{\mathrm{d}}

\theoremstyle{plain}
\theoremstyle{plain}\newtheorem{remark}{Remark}

\newcommand{\RR}{\mathbb{R}}
\newcommand{\In}{\textrm{in}}
\newcommand{\dt}{\Delta t}
\newcommand{\dx}{\Delta x}
\newcommand{\dv}{\Delta v}
\newcommand{\fij}{f_{i,j}}
\newcommand{\fijm}{f_{i-1,j}}
\newcommand{\fijp}{f_{i+1,j}}
\newcommand{\gij}{g_{i,j}}
\newcommand{\gijm}{g_{i-1,j}}
\newcommand{\gijp}{g_{i+1,j}}
\newcommand{\np}{{m+1}}

\newcommand{\var}{\textrm{Var}}
\newcommand{\bbE}{\mathbb{E}}

\newcommand{\cL}{\mathcal{L}}
\newcommand{\cP}{\mathcal{P}}
\newcommand{\cT}{\mathcal{T}}
\newcommand{\frakG}{\mathfrak{G}}

\newcommand{\numf}{\mathsf{f}}
\newcommand{\nume}{\mathsf{e}}
\newcommand{\numfo}{\mathsf{f}_1}
\newcommand{\numg}{\mathsf{g}}
\newcommand{\norm}[1]{ \| #1 \|}
\newcommand{\Lop}{{\mathcal L}}

\newcommand{\average}[1]{ \langle#1 \rangle}

\graphicspath{{Figures/}}

\title{Monte Carlo gradient in optimization constrained by radiative transport equation \thanks{\funding{Q.~Li is supported in part by NSF-DMS-1750488 and WARF-Wisconsin. L.~Wang is partially supported by NSF grant DMS-1846854. Y.~Yang acknowledges support from Dr.~Max R\"ossler, the Walter Haefner Foundation, and the ETH Z\"urich Foundation. All three authors have participated in ``Geometric Methods in Optimization and Sampling," held at the Simons Institute for the Theory of Computing in Fall 2021, and ``Frontiers in kinetic theory: connecting microscopic to macroscopic scale" held at the Isaac Newton Institute (UK) in Spring 2022, where the work was initiated. This work was also partially supported by a grant from the Simons Foundation.}}}
\author{Qin Li\thanks{Department of Mathematics, University of Wisconsin-Madison, Madison, WI (\email{qinli@math.wisc.edu}).}
\and Li Wang\thanks{Department of Mathematics, University of Minnesota Twin Cities, Minneapolis, MN (\email{wang8818@umn.edu}).}
\and Yunan Yang\thanks{Institute for Theoretical Studies, ETH Z\"urich, Z\"urich, Switzerland (\email{yunan.yang@eth-its.ethz.ch}).}
}

\headers{}{Qin Li, Li Wang, and Yunan Yang}

\date{\today}

\begin{document}

\maketitle
\begin{abstract}
Can Monte Carlo (MC) solvers be directly used in gradient-based methods for PDE-constrained optimization problems? In these problems, a gradient of the loss function is typically presented as a product of two PDE solutions, one for the forward equation and the other for the adjoint. When MC solvers are used, the numerical solutions are Dirac measures. As such, one immediately faces the difficulty in explaining the multiplication of two measures. This suggests that MC solvers are naturally incompatible with gradient-based optimization under PDE constraints. In this paper, we study two different strategies to overcome the difficulty. One is to adopt the Discrete-Then-Optimize technique and conduct the full optimization on the algebraic system, avoiding the Dirac measures. The second strategy stays within the Optimize-Then-Discretize framework. We propose a correlated simulation where, instead of using MC solvers separately for both forward and adjoint problems, we recycle the samples in the forward simulation in the adjoint solver. This frames the adjoint solution as a test function, and hence allows a rigorous convergence analysis. The investigation is presented through the lens of the radiative transfer equation, either in the inverse setting from optical imaging or in the optimal control framework. We detail the algorithm development, convergence analysis, and complexity cost. Numerical evidence is also presented to demonstrate the claims.
% In this paper, we design two Monte Carlo-based algorithms for computing gradients in the optimization tasks constrained by the radiative transport equation. We use the adjoint state method to treat the constraint and follow either the optimize-then-discretize or discretize-then-optimize approaches. At first sight, the Monte Carlo method is not suitable for gradient calculation as the gradient often depends quadratically on the solutions of the forward and adjoint equations. Thereby it is hard to make sense of the product of two empirical measures, which comes out as the representation of particle methods. We overcome this difficulty by solving the adjoint equation in tandem with the forward equation instead of evolving two independent sets of particles, and as a result, it largely mitigates the total computational cost and memory requirement. We also prove the convergence of the Monte Carlo solution as well as the Monte Carlo gradient, as the number of particles goes to infinity, in an appropriate weak topology. As an application, we consider both the inverse and optimal control problems and contrast our gradient with the one obtained by the conventional finite volume scheme. 
\end{abstract}


% 65K10 Numerical optimization and variational techniques
% 65C05,Monte Carlo methods
% 65M75 Probabilistic methods, particle methods, etc. for initial value and initial-boundary value problems involving PDEs
% 82C70  Transport processes in time-dependent statistical mechanics

\section{Introduction}

Monte Carlo (MC) method is a class of computational strategies that use random samples/particles to represent the underlying distributions. It is immensely popular in numerical problems with high dimensions due to its practical and theoretical advantages. It is simple to code up and converges robustly with a convergence rate independent of the dimension~\cite{caflisch1998monte} in certain metrics. One class of such high-dimensional problems is introduced by kinetic equations whose state variables reside in seven-dimensional phase space. Parallel to developing mesh-based method, Direct Simulation Monte Carlo (DSMC) and its extensions~\cite{bird1970direct, nanbu1980direct, bobylev2000theory,pareschi2001introduction,kratzer2009monte,pareschi2013interacting,huang2020mean} have garnered a great amount of interests.
%One high dimensional problem is introduced by kinetic equations whose state variables reside on a seven-dimensional phase space. Within the MC framework, Direct Simulation Monte Carlo (DSMC) and its extensions have been prevalent~\cite{bird1970direct, nanbu1980direct, bobylev2000theory,pareschi2001introduction,kratzer2009monte,pareschi2013interacting} in simulating kinetic equations. Despite the vast development in the deterministic mesh-based methods in the past decade, the Monte Carlo method has gained renewed interest lately due to the current growing need for machine learning-related tasks. 

PDE-constrained optimization is a classical formulation to solve inverse and control problems. In these problems, the parameters are adjusted accordingly to either fit the reference data or achieve a certain desired property. In the updating process, gradients of the loss function are computed in each iteration. This usually translates to computing two PDEs, one is the original forward equation and the other being the adjoint PDE. When these PDEs are presented on high dimensions, a prohibitive computational cost may incur.

To save numerical cost, we naturally face the option of utilizing MC solvers for gradient-based PDE-constrained optimizations. Then a natural question arises: can they be applied directly? We investigate this problem using kinetic equations as the underlying PDE. An overarching message we would like to deliver in this paper is that, at least in the case of kinetic equations,

\medskip
\begin{center}
\emph{MC solvers are not immediately compatible with gradient-based optimization strategies.}
\end{center}

\medskip
At the core of the difficulty is the incompatibility of the MC perspective and the computation of the gradients. Indeed, the MC solver views the PDE solution as a linear combination of Dirac delta functions, each representing one particle. The gradients derived for the PDE-constrained optimization problems, however, typically constitute a quadratic form that requires the multiplication of two PDE solutions (one forward and one adjoint). Mathematically, it is challenging to make sense of a multiplication of two delta measures. Physically, this incompatibility simply comes from the fact that two independent particles, formulated from forward and adjoint procedures, do not share trajectories, and there is an ambiguity in tracing the information. As such, a change of perspective is needed to integrate MC PDE solvers into the PDE-constrained optimization. Since MC solvers only provide convergence to the PDEs in the weak sense, we have to evaluate the convergence through the dual perspective, and all the rigorous proof must be done with exceedingly careful calculation in the observable space~\cite{babovsky1989convergence}.

This shift of perspective will be demonstrated in this paper through the lens of the radiative transfer equation (RTE), a classical kinetic equation~\cite{chandrasekhar2013radiative}. Kinetic equation is a class of equations describing the dynamics of many interacting particles. It is widely used in statistical mechanics to characterize the evolution of the distribution function of particles before it achieves the equilibrium state. It sits between thermo/fluid dynamics that describe macroscopic evolution and molecular dynamics that focus on fine-scale motions. Kinetic equations are practically useful in engineering and physics, and they also carry some unique mathematical features. Computationally, MC methods become a natural class of candidate due to the many-body nature of the system. Samples are drawn from the initial distribution and then moved around according to the physical law prescribed by the equation to simulate the evolution of the whole distribution. Monte Carlo strategy has been successfully applied to the study the RTE (for photons), the Boltzmann equation (for rarefied gas), and the Vlasov--Fokker--Planck equation (for plasma), among many others~\cite{bird1970direct,caflisch1998monte,hirvijoki2013monte,dimarco2018asymptotic}. In the optimization setting, various MC solvers are designed in~\cite{hochuli2015forward,powell2017radiance}. In our paper, we choose RTE because the equation is linear with a clear collisional spectrum, and the well-posedness of the associated inverse problem forms a well-founded base for us to focus on the algorithmic aspect~\cite{mccormick1992inverse,klose2002optical1,klose2002optical2,bal2009inverse,dimarco2018asymptotic,Chen_2018,Chen_2018_online}.

Our proposal consists of two strategies. One is the Optimize-Then-Discretize (OTD) approach. It starts with the original optimization problem and derives the gradient of the loss function. This gradient can usually be written as a product of two PDE solutions. To resolve the incompatibility between MC solvers and the gradient computation, we propose to, instead of directly using MC solvers to compute both forward and adjoint equations separately, simulate only the forward equation using MC, and record the trajectory for propagating the adjoint information backward in time. A similar strategy was applied to the Boltzmann equation constrained optimization in~\cite{caflisch2021adjoint}. Such algorithmic design is effortless to execute and will be proved to be theoretically sound, in the sense that it still honors the law of large numbers. This way respects the physical intuition for the role of the adjoint solver (propagating information back to the origin of change), and also avoids the artificial mathematical difficulty in understanding the product of two Dirac delta measures. These features are utilized in the rigorous justification of the computation method. 

The other approach falls in the Discretize-Then-Optimize (DTO) framework~\cite{betts2005discretize}. It discretizes the forward PDE first and presents the PDE solution using MC particles. The objective function and the constraints are represented only by these particles. Therefore, the full system is already discrete, and the adjoint variables are particles with a one-to-one correspondence to the MC particles from the forward PDE particle solution~\cite{caflisch2021adjoint}. The entirely discrete optimization problem removes the incompatibility issue mentioned above since the derived numerical scheme for the adjoint variables is always consistent with the forward discretization within this DTO framework. This MC gradient technique can be further applied to rejection sampling~\cite{Blei2017,mohamed2019monte,yang2022adjoint}, which also applies to the MC method for the forward RTE. We will leave rigorous proof of this approach for future research.
%\ql{should we say that we leave the rigorous proof for this approach to future project?}\yy{I have changed.}

Despite the differences, the OTD and DTO approaches designed in our work for computing the RTE gradient share some similar traits. First, both of them only need to simulate one set of particles for the forward RTE while the solution to the adjoint equation is obtained for free. Consequently, they have the same memory requirement. Second, the most expensive part of obtaining the gradient is computing the discrete integrals using disordered particles. 
% The OTD is slightly more costly as three integrals need to be calculated (see \eqref{eqn:fn}--\eqref{eqn:fgn}) compared to the one integral in the DTO approach (see \eqref{eq:DTO_grad}). 
These two approaches are equivalent under proper discretization schemes~\cite{hager2000runge,burkardt2002insensitive,caflisch2021adjoint}.
%Theoretically, the two approaches are also equivalent in the limit when the number of particles goes to infinity. Such equivalence has been presented in some earlier works
%\ql{need more} 

% \ql{are we really comfortable saying this?}
% \yy{I have changed.}

%\yy{I think they are the same cost in general. OTD needs for-loops over both space and time while DTO only needs for-loops over space. However, the sorting in DTO is $NM$ scale, $\mathcal{O}(NM (\log N + \log M))$, while the sorting in OTD is $M$ times $N$-scale, $\mathcal{O}(NM (\log N))$. The sorting in DTO can be reduced to $\mathcal{O}(NM (\log N))$ at the cost of for-loops over both space and time. Overall, they are the same.}

We now quickly summarize the equation and the setup in~\Cref{subsec:setup}. In~\Cref{sec:OTD}, we employ the OTD approach. We will review an MC solver for the RTE, present the failure of its direct extension to compute the gradient and propose our fix. The associated rigorous numerical analysis is presented in~\Cref{sec:NA}, including both the convergence of the MC solver for the forward RTE (see~\Cref{alg:f-RTE}) and the convergence of the gradient (see~\Cref{alg:P-OTD}).~\Cref{sec:DTO} is dedicated to the DTO approach, and~\Cref{alg:P-DTO} will be developed for computing the RTE-constrained optimization gradient within the DTO framework. Numerical evidences are presented in~\Cref{sec:tests}.
% to demonstrate the accuracy of the designed algorithms for gradient calculation with the one obtained from the finite-volume scheme as the reference. 
\subsection{Equation and Setup} \label{subsec:setup}
RTE is a model problem for simulating light propagation in an optical environment~\cite{chandrasekhar2013radiative}. For exposition simplicity, we restrict ourselves to the time-dependent RTE with no boundary effect,
\begin{equation}\label{eqn:RTE}
\begin{cases}
\partial_t f +v\cdot\nabla_xf &= \, \sigma(x)\mathcal{L}[f], \\
\quad f(t=0,x,v) &= \, f_\In (x,v),
\end{cases}\qquad x\in \RR^{d_x}, ~ v \in \Omega\,.
\end{equation}
Here, $f(t,x,v)$ is the distribution function of photon particles at time $t$ on the phase space $(x,v)$. The left-hand side of the equation describes the photon moving in a straight line in $x$ with velocity $v$, whereas the right-hand side characterizes the photon particles' interaction with the media characterized by the function $\sigma(x)$. The term $\mathcal{L}[f]$ writes:
% \begin{equation}\label{eq:f_v}
% \mathcal{L}[f] = \langle f\rangle_v - f\,,\quad\text{with}\quad \langle f\rangle_v = \frac{1}{|\Omega|} \int_\Omega f(x,v)\rd{v}\,.
% \end{equation}
\begin{equation}\label{eq:f_v}
\mathcal{L}[f] = \frac{1}{|\Omega|} \langle f\rangle_v - f\,,\quad\text{with}\quad \langle f\rangle_v =  \int_\Omega f(x,v)\rd{v}\,.
\end{equation}
The term $\sigma(x) \mathcal{L}[f]$ represents that particles at location $x$ have a probability proportional to $\sigma(x)$ to be scattered, into a new direction uniformly chosen in the $v$ space. As a result, the distribution function in the phase space exhibits a gain on $|\Omega|^{-1}\langle f\rangle_v(x)$ and a loss on $f(x,v)$. Throughout the paper, we use $\langle\cdot\rangle_\ast$ to denote the integration with respect to the variable $\ast$, e.g., $\langle\cdot \rangle_v$, $\langle\cdot\rangle_{xv}$ and $\langle\cdot\rangle_{txv}$ with the Lebesgue measure.

The forward problem~\eqref{eqn:RTE} has a unique solution under very mild conditions on both $\sigma(x)$ and the initial data $f_\In$; see for instance \cite{dautray1993mathematical,bal2006radiative}. %\cite[Section 3]{egger2016class}. \lw{better references?}
Moreover, the equation preserves mass because $\rho=\langle f\rangle_{xv}$ is a constant in time. Without loss of generality, we set $\rho=1$ throughout the paper.

RTE is widely used as the forward model in inverse/control problems. One example of an inverse problem sits under the general umbrella of optical imaging~\cite{klose2002optical2,powell2017radiance, herty2007optimal}. In an experiment, one shines light into a domain with unknown optical properties and measures the light intensity that comes out of the domain. The measurements are then used to infer the optical property of the domain interior. This inverse problem is widely used in medical imaging and remote sensing, albeit the frequency of light is adjusted accordingly~\cite{ren2006frequency}. Another example in control theory using RTE as the forward model is lens design~\cite{mashaal2016aplanatic}. By adjusting the optical properties of the manufactured lens, light can be bent in the desired manner. In both cases, it is the optical parameter in the RTE that needs to be determined. We now examine the corresponding computational methods to solve such problems.

We frame our problem using the approach of PDE-constrained optimization. The objective function, respectively, is the mismatch of the simulated light intensity and the measured data, or the desired light output, and the constraints come from the fact that the solution needs to satisfy a forward RTE. Without loss of generality, we assume the measurement is the final-time solution of the RTE,
\[
d(x,v) \approx f(t=T,x,v)\,,
\]
then the corresponding PDE-constrained optimization reads:
\begin{equation}\label{eqn:min}
% \begin{aligned}
\min_\sigma J(\sigma),\quad \text{with}\,\, J(\sigma) =\frac{1}{2}\Big\langle| f_\sigma(t=T,x,v) - d(x,v) |^2\Big\rangle_{xv},
% &\text{s.t.}& \partial_t f +v\cdot\nabla_xf = \sigma(x)\mathcal{L}[f]\,,\\
% && f(t=0,x,v) = f_\In(x,v)\,.
% \end{aligned}
\end{equation}
where $f_\sigma(t,x,v)$ is the simulated data that solves~\eqref{eqn:RTE} with the given initial condition and absorbing parameter $\sigma(x)$. 
%The form of objective function in~\eqref{eqn:min} is just one example while other proper functionals of $f_\sigma(t,x,v)$ could be used here.  
We will write $f_\sigma(t,x,v)$ as $f(t,x,v)$ hereafter for the simplicity of notation. The goal is to adjust $\sigma$ so that the simulated data is as close to the true measured data $d(x,v)$ as possible. Here we use the standard $L^2$ norm to measure ``closeness,'' but it shall be generalized to other metrics. In real-life problems, it is typical that the measurement is ``intensity'' ($\langle f\rangle_v$) instead of the solution profile $f$. Additionally, it is very typical to add a regularization term to mitigate the effects of noise and improve the convexity of the optimization problem. The derivation below can be extended easily to deal with all these variations.

The integration of MC methods and this constrained optimization problem will be presented from two perspectives in the rest of the paper. On the one hand, we can follow the OTD approach and derive the Fr\'echet derivative of $J$ with respect to $\sigma$. This will come down to simulating two PDEs (one forward and one adjoint RTE), and the MC solver needs to be properly applied. On the other hand, we also take the DTO approach and start by reformulating~\eqref{eqn:min} into a discrete form based on the MC solver. The gradient and optimization are then conducted entirely on this algebraic system. These two pathways will be presented in~\Cref{sec:OTD,sec:NA}, and~\Cref{sec:DTO}, respectively.

\section{Optimize-Then-Discretize Framework}\label{sec:OTD}
The OTD framework is the most straightforward numerical strategy for solving PDE-constrained minimization problems. This is to derive a Lagrangian to eliminate the constraints and adopt a minimization method, per users' choice, to handle the resulting unconstrained optimization problem. This minimization strategy is performed directly on the PDE level, and the formulation is written in a continuous setting. When gradient-based optimization methods are used, one typically needs to compute the Fr\'echet derivatives for updates. Discretization is then performed to simulate the PDE and approximate the Fr\'echet derivative for the final execution of the algorithm. We detail the procedure below.

We first apply the method of Lagrange multipliers and transfer the original constrained  optimization formulation~\eqref{eqn:min} into an unconstrained one:
\[
\min_{\sigma,f,g,\lambda}\mathfrak{L}(f,g,\lambda,\sigma)\,,
\]
where
\[
\mathfrak{L}(f,g,\lambda,\sigma) =  J(f) + \langle g\,,\partial_t f +v\cdot\nabla_x f -  \sigma(x)\mathcal{L}[f]\rangle_{txv} + \langle \lambda \,,f(t=0,x,v) - f_\In(x,v)\rangle_{xv}\,.
\]
Here, functions $g(t,x,v)$ and $\lambda(x,v)$ are Lagrange multipliers with respect to the RTE solution $f(t,x,v)$ for $t>0$ and the initial condition $f(t=0,x,v)$, respectively. It is common to refer to $f(t,x,v)$ as the state variable and $g(t,x,v)$ as the adjoint variable in PDE-constrained optimizations. With integration by parts, we rewrite the Lagrangian as:
\[
\begin{aligned}
\mathfrak{L}(f,g,\lambda,\sigma) = & J(f)  - \langle f\,, -\partial_t g +v\cdot\nabla_xg + \sigma(x)\mathcal{L}[g]\rangle_{txv} + \langle \lambda\,,f(t=0,x,v) - f_\In(x,v)\rangle_{xv}\\
& + \langle g(t=T,x,v)\,, f(t=T,x,v)\rangle_{xv} - \langle g(t=0,x,v)\,, f(t=0,x,v)\rangle_{xv} \,.
\end{aligned}
\]
Setting the variation of $\mathfrak{L}$ with respect to $f$ to be $0$, we obtain the adjoint equation:
% In the framework of PDE-constrained optimization, it is typical to set $\frac{\delta\mathfrak{L}}{\delta f}=0$ to derive the equations for the adjoint variables to ensure the optimality condition held, and then only descend in the $\sigma$ direction. To do so, we take the derivative of the Lagrangian $\mathfrak{L}$ with respect to $f$ and set the derivative to be $0$ for all $x,v,t$. This gives the adjoint equation:
\begin{equation}\label{eqn:g_eqn}
\begin{cases}
- \partial_tg  - v\cdot\nabla_xg &= \sigma(x)\mathcal{L}[g], \\
\qquad g(T,x,v) &=  - \frac{\delta J}{\delta f}(T,x,v)= d(x,v) - f(T,x,v)\,,
\end{cases}
\end{equation}
where the final condition at $T$ comes from the form of $J$ given in~\eqref{eqn:min}. Then the Fr\'echet derivative of $\mathfrak{L}$ with respect to the function $\sigma$ has the form:
% \begin{equation}\label{eqn:frechet_derivative}
% \frac{\delta\mathfrak{L}}{\delta\sigma}(x) = - \int_{v,t} f\mathcal{L}[g]\rd{v}\rd{t} = - |\Omega| \left( \int_t \langle f\rangle_{v}\langle g\rangle_{v}\rd{t} -   \int_t \langle fg\rangle_v\rd{t} \right),
% \end{equation}
\begin{align}\label{eqn:frechet_derivative}
\mathfrak{G}&:= \frac{\rd \mathfrak{L}}{\rd \sigma} 
= \frac{\delta \mathfrak{L}}{ \delta f} \frac{\rd f}{\rd \sigma } + 
\frac{\delta \mathfrak{L}}{ \delta g} \frac{\rd g}{\rd \sigma }  +
\frac{\delta \mathfrak{L}}{ \delta \lambda} \frac{\rd \lambda}{\rd \sigma }  + \frac{\delta\mathfrak{L}}{\delta\sigma} \nonumber
\\ &=
\frac{\delta\mathfrak{L}}{\delta\sigma}(x) = - \int_{v,t} f\mathcal{L}[g]\rd{v}\rd{t} = \int_t \langle fg\rangle_v\rd{t}  - \frac{1}{|\Omega|}\int_t \langle f\rangle_{v}\langle g\rangle_{v}\rd{t}  =: \mathfrak{G}_1 -   |\Omega|^{-1}\mathfrak{G}_2\,,
\end{align}
where the terms $\frac{\delta \mathfrak{L}}{ \delta f}$, $\frac{\delta \mathfrak{L}}{ \delta g}$ and $\frac{\delta \mathfrak{L}}{ \delta \lambda}$ vanish since $f$ and $g$ solve~\eqref{eqn:RTE} and~\eqref{eqn:g_eqn}, respectively. Here,  $\mathfrak{G}_1 =  \int_t \langle fg\rangle_v\rd{t}$ and $\mathfrak{G}_2 = \int_t \langle f\rangle_{v}\langle g\rangle_{v}\rd{t}$.

%\ql{what is this for?}We remark that the final condition of the adjoint equation~\eqref{eqn:g_eqn} is unique up to  a scalar multiplication as it depends on how the Lagrangian multiplier is added. Correspondingly, the derivative formula~\eqref{eqn:frechet_derivative} may also be different up to a constant scaling. However, the resulting derivative $\frac{\delta\mathfrak{L}}{\delta\sigma}(x)$ is always unique independent of the scaling constants in~\eqref{eqn:g_eqn} and~\eqref{eqn:frechet_derivative}. 
%\lw{maybe delete this paragraph?}\yy{I am fine with deleting it (I wrote it)}

The derivation above is carried out completely on the function space in the continuous setting, as done in the OTD framework. Here, we endow the functional space for the parameter $\sigma(x)$ with the $L^2$ inner product, so the derivative $\frac{\delta\mathfrak{L}}{\delta\sigma}(x)$ is indeed the $L^2$ gradient. As such, we use ``gradient'' and ``derivative'' interchangeably. Upon obtaining \eqref{eqn:frechet_derivative}, the next step is to find a discrete approximation to it. In particular, we represent the unknown parameter function $\sigma(x)$ using a finite-dimensional object, such as a piecewise polynomial function on a compact domain, and replace $f$ and $g$ by their associated numerical representations. As a result, the numerical error purely comes from the computation of~\eqref{eqn:RTE}, \eqref{eqn:g_eqn} and the derivative-assembling~\eqref{eqn:frechet_derivative}.

A handful of numerical strategies can be used to solve the constraint PDE~\eqref{eqn:RTE}, such as finite difference, finite element, and spectral methods~\cite{lewis1984computational, alldredge2012high,li2017implicit}. In this paper, we focus on the Monte Carlo method and investigate the possibility of using it to compute the gradient~\eqref{eqn:frechet_derivative}.

\subsection{Monte Carlo Method for $f$} \label{sec:MC-f}
In this subsection, we first review the MC method \cite{pareschi2001introduction,dimarco2018asymptotic} used in solving the forward problem \eqref{eqn:RTE}. It serves as the base of our computation. Though intuitive, the proof has not been thoroughly presented in the literature. We give a theoretical guarantee for the convergence of this method in this section.

Using MC to solve \eqref{eqn:RTE} amounts to representing the solution $f$ as an ensemble of many particles:
\begin{equation}\label{eq:f_n_approx}
f(t,x,v) \approx \frac{1}{N}\sum_{n=1}^N \delta\left(x- {x_n(t)} \right) \delta\left(v- {v_n(t)} \right)=: \numf_N(t,x,v)\,,
\end{equation}
where $(x_n(t),v_n(t))$ denote the $n$-th particle's location and velocity, respectively. Since~\eqref{eqn:RTE} has a natural interpretation of particle interaction, it is straightforward to single out the particle dynamics. More precisely, 
\begin{itemize}
    \item{Term $v\cdot \nabla_xf$:} It is a transport term suggesting that particles should move with velocity $v$, and therefore we set
    \[
    \dot{x}_n=v_n\,.
    \]
    \item{Term $\sigma(x)\mathcal{L}[f]$:} This term indicates that, with intensity $\sigma(x)$, particles interact with the media and adopt a new velocity, which is uniformly chosen from the velocity domain $\Omega$. This is a reminiscent of the Poisson process in that for all $s-t<p\sim \text{Pois}(e^{-\sigma(x_n)})$, $v_n(s)=v_n(t)$, and when $s=t+p$, $v_n$ switches to
    \[
    v_n(t+p) = \eta\sim \mathcal{U}(\Omega)\,,
    \]
    where $\mathcal{U}$ stands for a uniform distribution.
\end{itemize}
These understandings prompt the following formulation for~\Cref{alg:f-RTE}, see also~\cite{dimarco2018asymptotic}.

\begin{algorithm}
\caption{Monte Carlo Method for Solving the Forward RTE~\eqref{eqn:RTE}\label{alg:f-RTE}}
\begin{algorithmic}[1]
\State  \textbf{Preparation:} $N$ pairs of samples $\{(x_n^0,v_n^0)\}_{n=1}^N$, sampled from the initial distribution $f_\In(x,v)$; the total time steps $M$ and the time interval $\Delta t$ so that $T = M \Delta t$; and the parameter function $\sigma(x)$.
\For{$m=0$ to $M-1$}
\State Given $\{(x_n^m,v_n^m)\}_{n=1}^N$, set $x_n^{m+1} = x_n^m + \Delta t \,  v_n^m$, $n=1,\ldots,N$.
% \State Compute $\alpha_i^{n+1} = \exp(-\sigma(x_n^{m+1}) \Delta t)$.
\State Draw random numbers $\{p^{m+1}_n\}_{n=1}^N$ from the uniform distribution $\mathcal{U}([0,1])$.
\If{ $p_n^{m+1} \geq \alpha_n^{m+1} = \exp(-\sigma(x_n^{m+1}) \Delta t)$}
\State Set $v_n^{m+1} = \eta_n^{m+1}$ where $\eta_n^{m+1}\sim\mathcal{U}(\Omega)$.
\Else 
\State Set $v_n^{m+1} = v_n^{m}$.
\EndIf
\EndFor
\end{algorithmic}
\end{algorithm}

As written, all particles $(x^m_n,v^m_n)$ are independent of each other, so we drop the sub-index $n$, and only keep $m$ representing the time step. From each time step to the next, two random variables are involved. One is the rejection sampling conducted through $p^{m+1}$. The other is the uniform sampling of $\eta^{m+1}$. We denote the expectation with respect to these two variables as $\mathbb{E}_{p^{m+1}}$ and $\mathbb{E}_{\eta^{m+1}}$, respectively.

We will prove that~\Cref{alg:f-RTE} provides an accurate solver for~\Cref{eqn:RTE} in the weak sense. In particular, consider the test function set:
\begin{equation}\label{eq:test_fcts}
%  \Phi =  C_c^\infty(\mathbb{R}^d \times \Omega) \,.
 \Phi  = C_c^\infty(\mathbb{R}^{d_x} \times \Omega)\,,%\Big\{\phi\in  C_c^\infty(\mathbb{R}^d \times \Omega): \|\phi\|_{W^{2,\infty}_{x,v}} \leq L_\Phi \Big\}\,,
\end{equation}
we will show that the error,  when tested against any $\phi\in\Phi$, is well controlled  in the sense of both the expectation  and the law of large numbers (LLN). More specifically, let
\begin{equation}\label{eqn:error_phi}
    \mathsf{e}_{N,\phi}^m := \langle
    \numf^m-f(t^m,x,v)\,,\phi\rangle_{xv}=\frac{1}{N}\sum_{n=1}^N \phi(x^m_n, v^m_n) - \int f(t^m,x,v)\phi(x,v) \rd x \rd v \,,
\end{equation}
denote the error term. It can be viewed as a dual norm associated with the function space $\Phi$, which metricizes the weak convergence. We may use the short-hand notation $f(t^m)$ to denote $f(t^m,x,v)$ hereafter. %For our specific choice of $\Phi$, this norm can be viewed as a generalization 
The quantity defined in~\eqref{eqn:error_phi} is related to the flat norm defined in \cite[Example 8.8]{peyre2019computational}. In principle, the fact that $\mathsf{e}_{N,\phi}^m$ converges to zero for all $\phi\in\Phi$ as $N\rightarrow\infty$ is equivalent to the weak convergence of $\mathsf{f}^m$ to $f(t^m)$\,.

For different realizations of $\{x_n^m\,,v_n^m\}$, the value $\mathsf{e}_{N,\phi}^m$ changes accordingly, and thus is a stochastic process on a probability space spanned by random variables $\{p^m,\eta^m\}_{m=1}^M$, and hence naturally generates a filtration. To stress the $m$ dependence, we sometimes denote
%\ql{$\eta$ and $p$ uses upper-index for time.}
\[ %\label{eq:expectation_n}
\bbE^m  = \begin{cases}
    \mathbb{E}_{(x^0,v^0)}, & m = 0,\\
    \mathbb{E}_{(x^0,v^0)}\bbE_{\eta^1}\bbE_{p^1} \cdots \mathbb{E}_{\eta^{m}}\mathbb{E}_{p^{m}}, & m \geq 1\,,
    \end{cases}
\]
as taking the expectation up to the filtration at time $t^m = m\dt$. In the following, $\mathbb{E}^m[ \mathsf{e}^m_{N,\phi}]$ and $\mathbb{E}[\mathsf{e}^m_{N,\phi}]$ are used interchangeably. We will show:
\begin{theorem}\label{thm:f_conv_summary}
Let $\numf$ be produced by~\Cref{alg:f-RTE} with $T>0$, $M\in\mathbb{N}$ and $\dt = T/M$. Then for any $\phi\in\Phi$, the error $\mathsf{e}_{N,\phi}^M$:
\begin{itemize}
    \item is first order in time in expectation:
    \begin{equation} \label{eq:f_convergence_exp_summary}
    \bbE| \mathsf{e}_{N,\phi}^M|=\mathcal{O}(\dt )\,;
    \end{equation}
    \item has an exponential concentration bound in $N$. Namely, for all $\epsilon>0$:
\begin{equation} \label{eq:f_convergence_summary}
    \mathbb{P}\left(|\mathsf{e}_{N,\phi}^M | \geq \epsilon + \dt \right) \lesssim  \exp \left(-N\epsilon^2\right)\,.
\end{equation}
\end{itemize}
\end{theorem}

We have a few comments on~\Cref{thm:f_conv_summary}. Firstly, the notations $\mathcal{O}$ and $\lesssim$ hide a constant dependence. As expected, this constant depends on $T$, initial condition $f_\text{in}$ and $\phi$, but we stress that it does not depend on $M$ and $\Delta t$. We will make this constant dependence clear in the proofs of~\Cref{prop:expectation} and~\Cref{prop:f_convergence} in~\Cref{sec:NA}, which are dedicated to explaining the two bullet points in \Cref{thm:f_conv_summary} separately. Secondly, the conclusion above presents a strong contrast against traditional numerical methods such as finite difference or finite volume, where Lax Theorem requires both consistency and stability for the convergence, and the stability requirement usually poses conditions on $\Delta t$. The statement in our theorem holds true with $\Delta t$ not experiencing extra stability issue. Thirdly, the error is first order in time and $\mathcal{O}(1/\sqrt{N})$ in the number of particles. Note that the exponential form in~\Cref{eq:f_convergence_summary} is consistent with the prediction of the LLN. Indeed, for a small $\delta$,
\[
\exp \left(-N\epsilon^2\right)\leq \delta\quad \Longleftrightarrow \qquad  \epsilon>\sqrt{|\log\delta|/N }\,.
\]
Thus, the lower bound of $\epsilon$ honors the celebrated $\mathcal{O}(1/\sqrt{N})$ convergence rate for MC methods.
%\ql{pls check if these statements look right.}\yy{it looks right} \lw{I changed $\epsilon>|\log\delta|/ \sqrt{N}$ to $\epsilon> \sqrt{|\log\delta|/ N}$}

We stress that the error is quantified when tested against a test function $\phi$. This is inevitable since $\numf$ is a summation of delta measures, and the convergence has to be presented in the weak form.

\subsection{Pitfall of a Direct Monte Carlo for $g$}\label{subsec:unco-gn}
Due to the similarity between the forward \eqref{eqn:RTE} and adjoint \eqref{eqn:g_eqn} equations, it is natural to use the same MC method to obtain $g$. Indeed, let $\tau = T-t$, \eqref{eqn:g_eqn} rewrites:
\begin{equation}\label{eq:reverse RTE}
\partial_\tau g -v\cdot\nabla_xg = \sigma\mathcal{L}[g]\,,
\end{equation}
which is precisely the same as~\eqref{eqn:RTE} except for the flip of the sign in velocity. Therefore, if we denote $\{x_n,v_n\}_{n=1}^N$ to be a list of $N$ particles, then the same particle motion described in~\Cref{sec:MC-f} can be used here to produce a consistent algorithm. Like before, we can define 
\[
\numg_N(\tau, x,v)=\frac{1}{N}\sum_{n=1}^N\delta(x-x_n(\tau))\delta(v-v_n(\tau))\,.
\]
As with $\numf_N$, we expect that $\numg_N$ approximates $g(\tau,x,v)$ in the same way as in~\Cref{thm:f_conv_summary}, namely, $\langle g-\numg_N\,,\phi\rangle\sim0$ in the $\dt\to 0$ and $N\to\infty$ limit.

Together, we have $\numg_N\approx g$ and $\numf_N\approx f$. It is attempting to compute $\mathfrak{G}$ by  replacing $f$ and $g$ with their numerical approximation $\numf_N$ and $\numg_N$, respectively. The discrete version of \eqref{eqn:frechet_derivative} may take the form: 
\[
\mathfrak{G}_1= \int \langle f\rangle_v\langle g\rangle_v\rd{t} \approx \dt\sum_{m=1}^M\langle\numf^m_N\rangle_v\langle\numg^m_N\rangle_v\,,\quad \mathfrak{G}_2=\int_t \langle fg\rangle_v\rd{t} \approx \dt\sum_{m=1}^M\langle\numf^m_N\numg^m_N\rangle_v\,.
\]
However, this immediately yields a problem since both $\numf_N$ and $\numg_N$ are defined as delta measures. The definitions of both $\langle\numf^m_N\rangle_v\langle\numg^m_N\rangle_v$ and $\langle\numf^m_N\numg^m_N\rangle_v$ represent multiplications of two delta measures and thus cannot be justified mathematically. Indeed, according to~\Cref{thm:f_conv_summary}, the convergence of $\numf_N\to f$ and similarly $\numg_N\to g$ is achieved only in the weak sense. Then the product of two weak limits loses its precise definition. This suggests that naively computing the forward and adjoint equations using the standard MC solvers for assembling the gradient in~\eqref{eqn:frechet_derivative} cannot produce an accurate numerical approximation to the gradient.

We address that this incompatibility of MC solvers with the Fr\'echet derivative computation is not a mathematical artifact but is rooted in the physical meaning of forward and adjoint equations. At the core of computing the Fr\'echet derivative, the forward solver delivers the initial data to the final time by picking up the media information along the evolution. The Fr\'echet derivative captures the dependence of the final data on the media $\sigma(x)$ where the forward trajectories have visited. A small perturbation in $\sigma(x)$ will change the final data accordingly. Adjoint solvers allow us to trace back this change in the final data to the perturbation in media. The inconsistency problem of the MC solver comes from the fact that the forward solver and the adjoint solver use independent trajectories. While the forward solver delivers the change of media along certain trajectories to the change of the final data, the adjoint solver traces back such change along a totally different set of trajectories. There is a chance to conduct computation if the two sets of trajectories are close, but these events occur with an extremely small probability due to the independent nature of the trajectory generation process.

We believe there are numerical strategies to resolve this issue, for example,  approximating Dirac delta functions using Gaussian kernels with small support. Nevertheless, we expect that these numerical tricks will pose high requirements on the density of trajectories on the phase space and the size of $N$ and therefore call for careful manipulation of the PDE solvers. We do not pursue this direction in the current paper but instead target a systematic fix of the inconsistency issue between MC solvers and the Fr\'echet derivative computation.

\subsection{Gradient Calculation Revisited} 
In this work, we develop a new method for \eqref{eqn:g_eqn} to avoid ambiguity in defining the product of two delta measures. To begin with, 
we examine the relation between the adjoint and the forward variables. One observation is that the adjoint variable is designed to have certain quantities preserved in time:
% similar solver as in ~\Cref{alg:f-RTE} can be applied to solve~\eqref{eqn:g_eqn}. However, keeping track of two sets of particles can be very uneconomical. Since the ultimate goal is to compute the derivative~\eqref{eqn:frechet_derivative} rather than the actual solution of $g$, a more desirable approach is to recycle the particles used in solving the forward problem~\eqref{eqn:RTE} to compute the derivative. To this end, we derive a correlated approach where the adjoint solution is obtained almost for free.  
% Our starting point is the observation that
\begin{equation}\label{tfg}
\partial_t  \langle fg\rangle_{xv} \equiv 0\,,
\end{equation}
% This requirement is imposed to eliminate the undesired quantities from the integration-by-part procedure, and
which can be justified, in our particular setting, by comparing \eqref{eqn:RTE} multiplied by $g$ and \eqref{eqn:g_eqn} by $f$. % This relation is not unique from our specific problem, but rather generic and can be found in most PDE-constrained optimizations.
% \lw{is \eqref{tfg} designed to be this or comes out naturally as it is? } \yy{\eqref{tfg} comes naturally from the equation for $f$ and $g$. It is not always like this... I think the structure is unique to RTE or certain linear PDEs.}

This relation, when replacing $f$ by $\mathsf{f}_N$, becomes:
\[ %\label{eq:fngn}
    \sum_{n=1}^N g(t^m, x_n^m, v_n^m) =    \sum_{n=1}^N g( t^{m+1},x_n^{m+1}, v_n^{m+1})\,,
\]
which can be easily satisfied if we require
% up to time discretization error. Therefore, a sufficient algorithm for $g$ to satisfy~\eqref{eq:fngn} is to set
\[%\label{eq:gn_back}
g(t^m, x_n^m, v_n^m)  = g(t^{m+1}, x_n^{m+1}, v_n^{m+1})\,, \quad m= M-1,\ldots,0,
\]
with the final condition set to be:
\begin{align} \label{psi}
   g(T, x_n^M, v_n^M)= \psi(x^M_n,v^M_n) = -\frac{\delta J}{\delta f}(T,x_n^M,v_n^M)\,.
\end{align}
As a result, along every particle's trajectory of $\{(x_n^m\,,v^m_n)\}_{m=0}^M$, $g$ takes a constant value. In other words, if we denote $\numg^m_N(x,v)$ as the numerical solution approximating $g(t^m,x,v)$ using $N$ particles, then we have scattered value of $\numg^m_N(x,v)$ at the particle locations along the dynamics:
\begin{align}\label{gi}
   \numg^m_N(x,v)\big|_{(x_n^m,v_n^m)}  = \numg^m_N(x_n^m,v_n^m) = \numg_n := \psi (x_n^M, v_n^M)\, ,
\end{align}
for all $0\leq m \leq M$ and $1\leq n \leq N$.  For $(x,v) \not\in \{(x_n^M, v_n^M)\}_{n=1}^N$, we may define the value of $ \numg^m_N(x,v)$ through interpolation from $\{ \numg^m_N(x_n^m,v_n^m) \}_{n=1}^N = \{\numg_n\}_{n=1}^N$. 

We call this approach a {\it correlated} approach, due to the fact that $\numg$ and $\numf$ are not fully independent, as discussed in~\Cref{subsec:unco-gn}.

\begin{remark}\label{rmk:g_solver}
We should emphasize that the algorithm presented in~\eqref{gi} is derived merely from the relation~\eqref{tfg}, which is {\it not} equivalent to the original adjoint equation~\eqref{eqn:g_eqn}. Indeed, starting from~\eqref{tfg}, one can only tell that $g$ satisfies:
\[
\langle\partial_tg-v\cdot\nabla_x g-\sigma\cL g\,,f\rangle_{xv}=0\,.
\]
That is, the adjoint equation holds only when projected on $f$, the solution manifold to the forward equation~\eqref{eqn:RTE}. This is different from saying that $g$ is a weak solution to~\eqref{eqn:g_eqn}, which requires the above equation to hold for any test function in $\Phi$ instead of $f$ alone. Nevertheless, the solver \eqref{gi} indeed provides an approximation to~\eqref{eqn:g_eqn}, and the proof is omitted from the current paper.
% is sufficient for the derivative computation. This is because the derivative is calculated on the solution manifold of $f$, as required by the fact $\frac{\delta \mathsf{L}}{\delta f} = 0$ in \eqref{eqn:frechet_derivative}. We term this approach a {\it correlated} approach. \ql{I am actually not sure about the last sentence... Who's $\mathsf{L}$?}
%\lw{I changed a few sentences above. Please check.} \yy{looks good to me}
% Nevertheless, the solver~\eqref{gi} still correctly compute the adjoint equation -- one has to directly start from the equation instead of relying on the constraint~\eqref{tfg}. In~\Cref{sec:NA} where we present the numerical analysis solver, we will directly prove the computation of the gradient, but will give a comment on the proof for this algorithm to be consistent with the equation.
\end{remark}

With the numerical solution $\numg$ in hand, we now proceed to find numerical approximation to the Fr\'echet derivative~\eqref{eqn:frechet_derivative}. Based on the weak formulation of $f$ and that $\mathsf{f}$ is presented discrete-in-time, the Fr\'echet derivative is written in its weak semi-discrete-in-time form as well. After testing~\eqref{eqn:frechet_derivative} against a test function $\phi(x)$ and conducting the simple Riemann sum in time, we have:
\begin{align}\label{eqn:semi-grad}
    \mathfrak{G}_\phi:=\average{\frac{\delta\mathfrak{L}}{\delta\sigma}, \phi}_x 
    & =  \int \langle f(t^m, \cdot, \cdot)g(t^m, \cdot, \cdot),  \,\phi(\cdot )\rangle_{xv}\rd{t} - |\Omega|^{-1} \int \Big\langle\langle f(t^m, \cdot, \cdot)\rangle_{v}\langle g(t^m, \cdot, \cdot)\rangle_{v},\, \phi(\cdot)\Big\rangle_x\rd{t} \nonumber
    \\
    &\approx   \dt\sum_{m=1}^M \left[ \frakG_{1,\phi}^m - |\Omega|^{-1}\frakG_{2,\phi}^m  \right]\,,
\end{align}
where we used the notations
\begin{equation}\label{eqn:frakG_cont}
    \frakG_{1,\phi}^m := \langle f(t^m, \cdot, \cdot)g(t^m, \cdot, \cdot),\,  \phi(\cdot )\rangle_{xv}\,,
    \qquad 
    \frakG_{2,\phi}^m := \Big\langle\langle f(t^m, \cdot, \cdot)\rangle_{v}\langle g(t^m, \cdot, \cdot)\rangle_{v}, \,\phi(\cdot)\Big\rangle_x \,.
\end{equation}

At every discrete time $t^m$, according to~\Cref{alg:P-OTD}, $\numf^m_N$ approximates of $f(t^m,\cdot,\cdot)$ using $N$ particles, and $\numg^m_N$ records the value of $g(t^m,\cdot,\cdot)$ on these particle trajectories. As a consequence, the discrete version of~\eqref{eqn:frakG_cont} writes as
\begin{equation}\label{eqn:frakG_dis}
    \frakG_{N,1,\phi}^{m} = \frac{1}{N} \sum_{n=1}^N \phi(x_n^m)\, \numg_n\,,\quad \frakG_{N,2,\phi}^{m} = \frac{1}{N} \sum_{n=1}^N \phi(x_n^m) \, \average{\numg}_v(x_n^m) \,,
\end{equation}
and thereby, the final discrete Fr\'echet derivative, when tested on $\phi$, takes the following form as the $N$-particle approximation to~\eqref{eqn:semi-grad}:
\begin{equation}\label{eqn:num-grad}
      \mathfrak{G}_{N,\phi} :=\average{\frac{\delta\mathfrak{L}^N}{\delta\sigma}, \phi}_x 
   = \Delta t  \sum_{m=1}^M(\frakG_{N,1,\phi}^{m} - |\Omega|^{-1} \frakG_{N,2,\phi}^{m})\,.
\end{equation}

In the implementation, we need the strong form of the derivative evaluated on a given mesh. To do so, for every $\bar{x}$ on the mesh, we denote $\Delta x$ the mesh size and $Q(\bar{x}\,,\Delta x)$ the hypercube centered at $\bar{x}$ with side length $\Delta x$, and set $\phi=\mathds{1}_{Q(\bar{x}\,,\Delta x)}$, the corresponding indicator function. The strong form of the derivative then becomes:
\begin{equation}\label{eq:OTD_grad}
\mathfrak{G}_N(\bar{x})=\frac{\delta\mathfrak{L}^N}{\delta \sigma}(\bar x)=  \Delta t\sum_{m=1}^N\Big(\langle \numf^m_N \numg^m_N\rangle_v (\bar x) -  |\Omega|^{-1} \langle \numf^m_N\rangle_v (\bar x) \langle \numg^m_N\rangle_v (\bar x) \Big) \,,
\end{equation}
with each term computed as:
\begin{eqnarray}
 \langle \numf^m_N\rangle_v  (\bar x)  &=& \frac{1}{|\Delta x|^d} \frac{1}{N}\sum_{n=1}^N \mathds{1}_{x_n^m \in Q(\bar x; \Delta x)} \,, \label{eqn:fn}\\
 \langle \numg_N^m\rangle_v (\bar x)   &=&   \sum_{x_{n_1}^m \in Q(\bar x; \Delta x)}  w_{n_1}^m \, \numg_{n_1} \,, \label{eqn:gn}\\
  \langle \numf_N^m \numg_N^m\rangle_v (\bar x)   &=& \frac{1}{|\Delta x|^d}\frac{1}{N} \sum_{n=1}^N \mathds{1}_{x_n^m \in Q(\bar x; \Delta x)} \,\numg_n \,  \label{eqn:fgn}.
\end{eqnarray}
Here, $\{w_{n_1}^m\}$ are the quadrature weights in the $v$ domain such that
\[
\sum_{x_{n_1}^m \in Q(\bar x; \Delta x)}  w_{n_1}^m \, \numg_{n_1} = \sum_{x_{n_1}^m \in Q(\bar x; \Delta x)}  w_{n_1}^m \numg_N^n(x_{n_1}^m, v_{n_1}^m) \approx \sum_{x_{n_1}^m \in Q(\bar x; \Delta x)}  w_{n_1}^m g(t^m, \bar x, v_{n_1}^m) \approx \int g(t^m, \bar x, v) \rd v.
\]
% These definitions are in place as approximations to Dirac delta measures. Indeed, as $\Delta x\to0$:
% \begin{equation}\label{eqn:limit_grad}
% \langle \numf^m_N\rangle_v(x)=\frac{1}{N}\sum_{n=1}^N\delta_{x_n^m}\,,\quad \langle \numf^m_N\numg^m_N\rangle_v(x)=\frac{1}{N}\sum_{n=1}^N\delta_{x_n^m}\numg_i\,,\quad \text{and}\quad \langle\numg^m_N\rangle_v(x) = \sum_n\langle\numg\rangle_v(x^m_n)\mathbb{I}_{x_n^m}\,.
% \end{equation}
% \ql{not quite sure about the last term... this is what we had before...``with $\average{\numg}_v(x_n^m)$ being the numerical integral of $g$ with respect to $v$ using trapezoidal rule as in \eqref{eqn:gn} (with $\bar x$ replaced by $x_n^m$)."}
In particular, out of three calculations listed in \eqref{eqn:fn}--\eqref{eqn:fgn}, \eqref{eqn:gn} is the most expensive part. The reason is that, at each time step,  one needs to loop over space to sort the particles, which results in  $\mathcal{O}(NM (\log N))$ total complexity with $N$ being the total number of particles and $M$ the total time steps. 

In~\Cref{alg:P-OTD}, we summarize the steps of computing the gradient using the correlated approach.

\begin{algorithm}
\caption{A Correlated Approach for Gradient Computation\label{alg:P-OTD}}
\begin{algorithmic}[1]
\State  Given cell centers $\{\bar{x}_j\}$ where we want to evaluate the gradient~\eqref{eqn:frechet_derivative}.
\State Implement~\Cref{alg:f-RTE} to solve the forward RTE~\eqref{eqn:RTE} and store the trajectories $\{(x_n^m,v_n^m)\}$ for all $n=0,\ldots,M$ and $n=1,\ldots,N$.  
\State Using~\eqref{eqn:min} and \eqref{psi} to determine the final condition for $g$ and set $\numg_n$ following~\eqref{gi}, for all $1\leq n\leq N$.
\For{$m=1$ to $M$}
\State 
%Use the stored $\{(x_n^m,v_n^m)\}_{n=1}^N$ and $\{\numg_n\}_{n=1}^N$ to compute  $a_j^m:=  \langle \numf^m_N \rangle_v  ({\bar x_j})$, $b_j^m:= \langle \numg^m_N \rangle_v  ({\bar x_j})$ and $c_j^m := \langle \numf^m_N \numg^m_N \rangle_v  ({\bar x_j})$ following~\eqref{eqn:fn},\eqref{eqn:gn} and~\eqref{eqn:fgn}, respectively, for every element in $\{\bar{x}_j\}$.
Use the stored $\{(x_n^m,v_n^m)\}_{n=1}^N$ and $\{\numg_n\}_{n=1}^N$ to compute  $  \langle \numf^m_N \rangle_v  ({\bar x_j})$, $\langle \numg^m_N \rangle_v  ({\bar x_j})$ and $\langle \numf^m_N \numg^m_N \rangle_v  ({\bar x_j})$ following~\eqref{eqn:fn},\eqref{eqn:gn} and~\eqref{eqn:fgn}, respectively, for every element in $\{\bar{x}_j\}$.
% \ql{can we remove the notations $a/b/c$?}\yy{done}
\EndFor 
\State The gradient evaluated at $\bar{x}_j$ is computed following~\eqref{eq:OTD_grad}.
%$\Delta t\sum_{m=1}^M\big( c_j^m -   a_j^m\, b_j^m / |\Omega|   \big)$.
\end{algorithmic}
\end{algorithm}

We note that the approximation to the Fr\'echet derivative~\eqref{eq:OTD_grad} contains four layers of error:
\begin{itemize}
    \item Time discretization: $\dt\sum_{m=1}^M$ is used as a replacement of $\int\rd{t}$. This is the simplest Riemann sum for the time integration, and we expect the error at the order of $\dt$. Throughout our analysis, this part of the error is omitted
    % and we do not distinguish semi-discrete-in-time solution and continuous-in-time solution, by setting
    % \[
    % \mathfrak{G}_\phi=\dt\sum_{m=1}^M[\mathfrak{G}_{1,\phi}(t^m)-|\Omega|^{-1}\mathfrak{G}_{2,\phi}(t^m)]
    % \]
    as we regard the following semi-discrete in-time Fr\'echet derivative as the ground truth:
    \begin{equation}
% \begin{aligned}
\frac{\delta\mathfrak{L}}{\delta\sigma}(x) =   \Delta t  \sum_{m=1}^M \Big(   \langle fg (t^m, x,\cdot )\rangle_v  - |\Omega|^{-1} \langle f(t^m,x,\cdot)\rangle_{v}\langle g(t^m,x,\cdot) \rangle_{v} \Big)\,.\label{eqn:frechet_derivative_semi}
\end{equation}
It is the strong form of~\eqref{eqn:semi-grad}.
% Here $f(t^m)$ and $g(t^m)$ are true solutions to the forward RTE \eqref{eqn:RTE} and the adjoint~\eqref{eqn:g_eqn} at time $t^m$. 
\item Spatial discretization. This is used in the final assembling of the gradient~\eqref{eq:OTD_grad} where for every fixed $\bar{x}$, a small hypercube of volume $\Delta x^d$ is drawn, and values are averaged out inside the cube. It naturally brings a smoothing effect. Such smoothing effects are heavily studied in the literature, and we only cite~\cite{raviart1985analysis} for details.
\item Quadrature error in the velocity domain. This enters through the term~\eqref{eqn:gn} to approximate the velocity domain integration $\langle \cdot \rangle_v$. The error from this term highly depends on the quadrature rule selected for defining $\{w_{n_1}^m\}$. With the simplest Trapezoidal rule, we expect the error to be of order $\mathcal{O}(\Delta v^2)$ where $\Delta v$ is the largest discrepancy between two sampled particles in velocity in the same $\bar{x}$-neighborhood. Throughout the paper, we denote this part of the error by $e_v$. We comment that this term inherits the randomness from MC sampling and requires some delicate analysis. This goes beyond the current scope of the paper and will be left to future study.
\item Monte Carlo error. This comes from the fact that random particles represent the PDE solution. \Cref{thm:f_conv_summary} states that $\mathsf{f}^m_N$ is a good approximation, with LLN ($\mathcal{O}(1/\sqrt{N})$) convergence rate, to the ground truth $f(t^m)$. A statement of similar flavor shall be provided for the gradient.
% Monte Carlo error. This comes from the fact that random particles represent the PDE solution. Theorem~\ref{thm:f_conv_summary} states that $\mathsf{f}^m_N$ is a good approximation, with LLN ($\mathcal{O}(1/\sqrt{N})$) convergence rate, to the ground truth $f(t^m)$. The same statement should be provided for~\eqref{eq:OTD_grad}, as a good approximation to~\eqref{eqn:frechet_derivative}. The paper focuses on the MC error while eliminating the $\dt$ and $\dx$ pollution. The estimates we obtain justify that in the $\dx\to0$ limit,~\Cref{alg:P-OTD} using~\eqref{eq:OTD_grad} indeed provides, with high confidence, an accurate approximation to the semi-discrete-in-time ground-truth gradient~\eqref{eqn:frechet_derivative_semi}.
\end{itemize}

Among these four types of errors, the second and the third kind start to matter only when the error is presented in the strong form. If the error is studied in the weak form, one only needs to analyze the first and the forth kinds. Since the time discretization is a standard Riemann sum error, the focus of the paper is placed on the forth: the MC error. The estimates we obtain justify that in the weak sense,~\Cref{alg:P-OTD} indeed provides, with high confidence, an accurate approximation to the semi-discrete-in-time ground-truth gradient.
% \ql{pls check}\yy{The number of particle $N$ needs to go to $\infty$ as $\dx \to 0$. In this end, it is the particle-per-cell that determines convergence.} \lw{So should we just change $\Delta x \rightarrow 0$ to $N \rightarrow \infty$ since we do not really consider spatial discretization in the paper? }
The rigorous statement is summarized as follows:
\begin{theorem}\label{thm:gradient_summary}
%Let $\mathfrak{G}_N$ be the numerical Fr\'echet derivative computed using~\Cref{alg:P-OTD}. \yy{$\mathfrak{G}_N$ defined in~\eqref{eq:OTD_grad} is not used in the statement of this theorem.} \lw{please change it as you see fit.} It is the numerical strong form of $\mathfrak{G}_{N,\phi}$~\eqref{eqn:num-grad}. 
For all $\phi\in\Phi$,  $\mathfrak{G}_{N,\phi}$ defined in~\eqref{eqn:num-grad} approximates the true derivative $\mathfrak{G}_\phi$ in its semi-discrete form~\eqref{eqn:semi-grad} with high probability. Namely, for any $\epsilon>0$,
\begin{equation}\label{eqn:prop_error_gradient}
\mathbb{P}(|\mathfrak{G}_{N,\phi}  - \mathfrak{G}_{\phi}|> \mathcal{O}(\epsilon+ \dt + |e_v|)) \leq 2\exp\left(-\frac{N\epsilon^2}{3C_3}\right)\,,
\end{equation}
with $C_3$ only depending on the regularity of the initial condition in~\eqref{eqn:RTE}, the final condition in~\eqref{eqn:g_eqn}, and the test function; see \eqref{C5}. % \ql{who's measurement...?}\yy{I have changed. Pls check too} \lw{looks good to me.}
\end{theorem}
This theorem is the theoretical justification of~\Cref{alg:P-OTD}. It states that the numerical Fr\'echet derivative, defined using a summation of delta measures, approximates the true gradient in the weak sense,
%when measured by the negative-Sobolev norm, 
with a probability that depends on the number of particles in the algorithm. Involving more particles guarantees a higher probability of capturing the gradient with better precision. The formula~\eqref{eqn:prop_error_gradient} gives a precise quantification of the decay of possibilities, in $N$, of getting outliers. Similar to~\Cref{thm:f_conv_summary}, this convergence rate is in line with the LLN prediction and obtains a $\frac{1}{\sqrt{N}}$ decay. Since this is the rate of the MC solver for the forward problem (see~\Cref{thm:f_conv_summary}), it is the best one can hope for in the computation of the gradient. 
%\ql{maybe add something about $N$}

\section{Numerical Analysis}\label{sec:NA}
We hereby provide the justification to~\Cref{thm:f_conv_summary,thm:gradient_summary} regarding the computation of $f$ in~\Cref{eqn:RTE} and the Fr\'echet derivative~\eqref{eqn:frechet_derivative} calculation.

\subsection{Convergence of the Monte Carlo Method for the Forward RTE}\label{subsec:fn}
In~\Cref{alg:f-RTE}, we have presented the MC solver for computing~\eqref{eqn:RTE}. From each time step to the next, two random variables are involved. One performs the rejection sampling, and the other is the uniform sampling for the new velocity direction. The current section is dedicated to proving~\Cref{thm:f_conv_summary} that gives both an expectation and concentration bound for the error $\mathsf{e}_{N,\phi}^m$.

Based on the definition in~\eqref{eqn:error_phi}, $\mathsf{e}^m_{N,\phi} =\langle\mathsf{f}_N^m-f(t^m)\,,\phi\rangle_{x,v}$ where $\mathsf{f}_N=\frac{1}{N}\sum_{n}\delta(x-x_n^m)\delta(v-v_n^m)$. Since $\{(x^m_n\,,v^m_n)\}_{n=1}^N$ are independent of each other for any fixed $m$, we have
\begin{equation}\label{eqn:N_error}
\mathsf{e}^m_{N,\phi}=\frac{1}{N}\sum_{n=1}^N\mathsf{e}^m_{1,\phi}(n)\,,
\end{equation}
where $\mathsf{e}^m_{1,\phi}(n)$ is the $n$-th realization of $\mathsf{e}^m_{1, \phi}$. This then implies that at the expectation level, $\mathbb{E}[\mathsf{e}^m_{1,\phi}]=\mathbb{E}[\mathsf{e}^m_{N,\phi}]$. At the variance level, the concentration inequality shall be applied. These arguments clearly suggest the following roadmap:

\begin{itemize}
    \item We will first prove that each particle, in expectation, solves the RTE~\eqref{eqn:RTE}. This is to say that, if $N=1$ and we accordingly define:
\begin{equation} \label{fn}
    \numfo^m(x,v) :=\delta(x-x^m)\delta(v-v^m)\,,
\end{equation}
then $\numfo^m(x,v)$ numerically solves~\eqref{eqn:RTE} in the weak sense in expectation, up to a discretization error in $\dt$.  More precisely, for any $\phi\in \Phi$, 
\[ %\label{bbEn}
\mathbb{E}^m [\mathsf{e}^m_{1,\phi}] =\mathbb{E}^m \left[ \langle \numfo^{m} - f(t^m),\phi\rangle_{x,v} \right] = \mathcal{O}(m \Delta t^2 ).
\]
The analysis is presented in~\Cref{sec:single_particle}, and it concludes the first part of~\Cref{thm:f_conv_summary}.
\item We then level the calculation up to the $N$-particle system using~\eqref{eqn:N_error}, calling the Bernstein inequality. To directly utilize the concentration bound, we will control the variance of $\mathsf{e}_{1,\phi}^m$. This is to be discussed in~\Cref{sec:many_particles} and it concludes the second part of~\Cref{thm:f_conv_summary}.
\end{itemize}

\subsubsection{Single Particle}\label{sec:single_particle}
We prove that each particle sampled according to~\Cref{alg:f-RTE}, in expectation, traces the evolution of $f$ in the weak sense. For any $\phi\in\Phi$, we recall the definition of the error in~\eqref{eqn:error_phi} and denote the expected value to be
\begin{equation}\label{eqn:def_error_exp}
    e^m_{1,\phi} = \mathbb{E}^m \left[\mathsf{e}^m_{1,\phi} \right]=\mathbb{E}^m \left[ \langle \numf^{m}_1 - f(t^m),\phi\rangle_{x,v} \right]\,,
\end{equation}
where the subindex $1$ in $\mathsf{e}$ and $e$ reflects that there is only one particle. This index will be omitted in the later part of this section.

We will track the growth of $e^m_\phi$ with respect to $m$ and prove that in each time step, the growth is controlled by $\mathcal{O}(\dt^2)$. Thus, the whole scheme is first-order accurate in time. In the proof, we use several positive constants denoted by $C_i$, $i=1,2,\ldots$. We will make explicit the constants' dependence on different parameters but do not spell out the specific dependence. Their exact values may change from line to line.

\begin{lemma}\label{lem:expectation}
Let $\sigma(x) \in W^{1,\infty} (\mathbb{R}^{d_x})$ and $\Omega$ be a bounded velocity domain. Denote by $\{(x^m\,,v^m)\}$ the solution of the  particle trajectory through~\Cref{alg:f-RTE}. Then $\mathsf{f}_1$ defined in~\eqref{fn} approximates $f$, the solution to~\eqref{eqn:RTE}, in the expectation sense. More specifically, for any $\phi\in \Phi$ and any $m\geq 0$, we have
\begin{equation} \label{eq:induction1}
 e^{m+1}_\phi \,\leq \, e^m_{(I+\dt \cP)\phi}  + C_1 \dt^2\,,\quad \cP \phi:= v\cdot \nabla_x \phi + \sigma\cL[\phi]\,,
\end{equation}
where the positive constant %\norm{\phi}_{W_{x,v}^{2,\infty}
\begin{equation} \label{C1}
C_1= C_1(\norm{f_\In}_{W^{2,\infty}_x L_v^\infty}, \|v\|_{L^\infty(\Omega)},  \norm{\sigma}_{W_x^{1,\infty}}, \norm{\phi}_{W_{x,v}^{2,\infty}} ,T )\,.
% \quad\text{and}\quad C_2 = 2 \|\sigma\|_{L_x^\infty} \,,
\end{equation} 
Moreover, $e_\phi^0 = 0$, and we have
\[ %\label{eq:induction}
|e^{m}_\phi| \leq m\, C_1 \dt^2  \,.
\]
Therefore,  $e^m_\phi\to0$ in the limit of $\dt\to0$ for every fixed $1\leq m\leq M$.

% \lw{should we use equality or inequality here? Now can we say that $e_\phi^m$ is positive? }
% \lw{I changed to the following:
% \begin{equation}  \label{e-0911}
%  e^{m+1}_\phi \,= \, e^m_\phi +  \dt \, e^m_{\cP \phi} + \dt^2 \Dop \phi \,,\quad \cP \phi:= v\cdot \nabla \phi + \sigma\cL[\phi]\,,
% \end{equation}
% where 
% \begin{equation} \label{Dphi}
%     \Dop \phi = \half {v^m}^T \nabla_x^2 \phi v^m + C(\nabla_x \sigma) \Lop \phi + C(\sigma) \nabla_x \Lop \phi + C(f) \phi \,,
% \end{equation}
% and we used $C(\nabla_x \sigma)$, $C(\sigma)$ to denote three quantities that depend on $\nabla_x \sigma$, $\sigma$ and  $\norm{f_\In}_{W_x^{2,\infty}L_v^\infty}$ respectively. }
\end{lemma}
\begin{proof}
According to the definition, 
\[
\mathsf{e}^m_\phi = \langle\mathsf{f}^m_1-f(t^m)\,,\phi\rangle_{xv} = \phi(x^m,v^m)-\langle f(t^m)\,,\phi\rangle_{xv}\,.
\]
To study $\mathsf{e}^m_\phi$, we need to track the evolution of $\mathsf{e}^{m+1}_\phi-\mathsf{e}^m_\phi$, which amounts to evaluating $\phi(x^{m+1}\,,v^{m+1})-\phi(x^m,v^m)$ and $\langle f(t^{m+1})-f(t^m)\,,\phi\rangle_{xv}$.

The term $\phi(x^{m+1}\,,v^{m+1})-\phi(x^m,v^m)$ can be expanded as
\begin{align} 
    & \phi(x^\np, v^\np) - \phi(x^m, v^m) \nonumber \\ 
    =&  \phi(x^\np, v^\np) - \phi(x^\np, v^m) + \phi(x^\np, v^m) - \phi(x^m, v^m)   \nonumber \\
    = & \phi(x^\np, v^\np) - \phi(x^\np, v^m)  + \Delta t\, v^m \cdot \nabla_x \phi(x^m, v^m)  +  \frac{ \dt^2}{2} (v^m)^\top \nabla_x^2 \phi(\xi_1) v^m \,,\label{0719}
    % + o(\dt^2)
\end{align}
where $ \nabla_x^2 \phi$ is Hessian of $\phi$ with respect to the $x$ direction, $\xi_1$ is between $x^m$ and $x^{m+1}$. 
%evaluated at a point $\xi$ and is known to be controlled by $\norm{\phi}_{W_{x,v}^{2,\infty}}$.
Note that the event $v^{m+1}=\eta\neq v^m$ occurs with probability $1-e^{-\sigma(x^{m+1})\Delta t}$.  Thus,
\begin{align}
& \mathbb{E}_{\eta^\np}\mathbb{E}_{p^\np}\left[\phi(x^\np, v^\np) - \phi(x^\np, v^m)\right] \nonumber \\
= &(1- e^{-\sigma(x^{m+1})\Delta t}) \mathcal{L}[\phi](x^{m+1} ,v^m) \nonumber \\ 
=  & \dt \left(  \sigma\mathcal{L}[\phi] \right)\big|_{(x^{m} ,v^m)}
+ \dt^2 \,  \left( \nabla_x \left(  \sigma\mathcal{L}[\phi] \right) \cdot  v\right)\big|_{(\xi_2 ,v^m)}\,,
\label{eq:0911}
% \dt \left(  \sigma\mathcal{L}[\phi] \right)\big|_{(x^{m} ,v^m)}  + \dt^2 \,  \left( \nabla_x \left(  \sigma\mathcal{L}[\phi] \right) \cdot  v\right)\big|_{(x^{m} ,v^m)}  + o(\dt^2)\,, 
\end{align}
where $\xi_2$ is again between $x^m$ and $x^{m+1}$. % \yy{There is an $\mathcal{O}(\dt^2)$ term from $ 1- e^{-\sigma(x^{m+1})\Delta t}$ to $\sigma(x^{m+1}) \dt$ that cannot be expressed by the MVT above }
% where we have used the fact that 
% \begin{align*}
% 1- e^{-\sigma(x^{m+1})\Delta t} &\leq   \sigma(x^{m+1}) \dt   = \sigma(x^{m}) \dt + \nabla_x \sigma(x^{m}) \cdot v^m   \dt^2 + o(\dt^2)\,. 
% \end{align*}

%\yy{I think the above needs to be changed slightly, because of $\sigma(x^{m+1})$ instead of $\sigma(x^{m})$, but I don't think it is a big problem.} \lw{I have changed. Please check.} 
Consequently, the above equation~\eqref{eq:0911}, along with~\eqref{0719}, leads to 
\begin{align}
       &  \bbE^{m+1} \left[ \phi(x^{m+1},v^{m+1})- \phi(x^m,v^m)   \right]\nonumber \\
    %   = & {\color{cyan}\dt \, \bbE^{m} \left[   \cP \phi (x^m,v^m) + \dt^2 \Dop_0 \phi \right]}
    %   \\
     \leq  & \dt \, \bbE^{m} \left[   \cP \phi (x^m,v^m) \right] + \dt^2  \left( \mathfrak{c}_1 \| \nabla_x \left(  \sigma\mathcal{L}[\phi] \right) \|_\infty +  \mathfrak{c}_2  \| \nabla_x^2 \phi  \|_\infty  \right) \,, 
     \label{eq:f_n expand}
\end{align}
where $\mathfrak{c}_1 = 2\|v\|_{L^\infty(\Omega)}$, $\mathfrak{c}_2 =  \|v\|_{L^\infty(\Omega)}^2$, and the operator $\cP$ is defined in~\eqref{eq:induction1}.
% \lw{I will not use this inequality but only the equality part. where $\Dop_0$ is 
% \[
% \Dop_0 \phi = \half {v^m}^T \nabla_x^2 \phi v^m + C(\nabla_x \sigma) \Lop \phi + C(\sigma) \nabla_x \Lop \phi \,.
% \].}

To study $\langle f(t^{m+1})-f(t^m)\,,\phi\rangle_{xv}$, we employ~\eqref{eqn:RTE}:
\[
    \frac{f(t^{m+1}) - f(t^m)}{\dt} = \sigma\, \cL [f](t^m) + v\cdot \nabla_x f(t^m) + \frac{\Delta t}{2} \partial_{tt}f(\xi_t)\,, \qquad \text{for} \quad \xi_t \in (t^m,t^{m+1})\,.
\]
Since $\partial_{tt} f$ solves the same equation \eqref{eqn:RTE} with initial condition
\[
\partial_{tt} f(0,x,v) = \sigma \mathcal L [\sigma \mathcal L [f_\In] - v \cdot \nabla_x f_\In] - v \cdot \nabla_x (\sigma \Lop [f_\In] - v \cdot \nabla_x f_\In)\,,
\]
it follows that 
%\yy{We need $\|\partial_{tt}f\|_{L^1_{x,v}}$ to be bounded}
\begin{equation} \label{cf}
\|\partial_{tt}f(t,\cdot, \cdot)\|_{L^1_{x,v}} \leq C \norm{f_\In}_{W^{2,\infty}_x\, L_v^\infty} := \mathfrak{c}_3\,,
\end{equation}
where $C$ and $\mathfrak{c}_3$ are constants.
Thus, we have
\begin{equation}\label{eq:f_tn expand}
    \langle f(t^{m+1}),\phi \rangle_{xv}\geq \langle f(t^m),\phi \rangle_{xv} + \dt\, \langle f(t^m),\, \cP \phi \rangle_{xv} - \dt^2\, \mathfrak{c}_3 \|\phi\|_\infty \,.
\end{equation}
% \lw{I will change to
% \begin{equation}
%      \langle f(t^{m+1}),\phi \rangle_{xv} = \langle f(t^m),\phi \rangle_{xv} + \dt\, \langle f(t^m),\, \cP \phi \rangle_{xv} + \dt^2\, C(f) \phi \,.
% \end{equation}
% and do not define $\mathfrak{c}_3$ as the upper bound. 
% Subtracting \eqref{eq:f_tn expand} from~\eqref{eq:f_n expand}, we have \eqref{e-0911} \eqref{Dphi}.
% }

Combine~\eqref{eq:f_tn expand} with~\eqref{eq:f_n expand}, we obtain that for any $n\geq 0$,
\begin{eqnarray} 
 e^{m+1}_\phi  &\leq & e^{m}_{(I + \dt \mathcal{P}) \phi }  + \dt^2  \left( \mathfrak{c}_1\|\cT_1 \phi \|_\infty +  \mathfrak{c}_2\|\cT_2 \phi \|_\infty +
 \mathfrak{c}_3\|\cT_3 \phi \|_\infty \right)\,,
 \label{eq:induction0}
\end{eqnarray}
where  the operators $\{\cT_i\}$ are defined as
\begin{equation} \label{eq:T}
    \cT_1 \phi :=\nabla_x  \left( \sigma  \mathcal{L}[\phi] \right)\,,\quad
    \cT_2 \phi := \nabla_x^2  \phi\,,\quad 
    \cT_3 \phi := \phi\,.
\end{equation}
A key observation is that $\cT_i (\cP \phi) =\cP (\cT_i  \phi)$, $1\leq i \leq 3$.

Since~\eqref{eq:induction0} applies to any $m=0,1,2,\ldots,M-1$, we have that 
\begin{eqnarray}
  e^{m+1}_\phi &\leq &   e^{m}_{(I + \dt \mathcal{P}) \phi } +  \dt^2  \sum_{i=1}^3 \mathfrak{c}_i \|\cT_i \phi\|_\infty \nonumber \\
  &\leq & e^{m-1}_{(I + \dt \mathcal{P})^2\phi }  +  \dt^2   \sum_{i=1}^3 \mathfrak{c}_i \left(  \|\cT_i \phi\|_\infty  +  \|(I+\dt \cP)  (\cT_i \phi) \|_\infty \right) \nonumber \\
  &\leq & e^{m+1 - k}_{(I + \dt \mathcal{P})^k \phi } + \dt^2  \sum_{i=1}^3 \mathfrak{c}_i \sum_{q = 0}^{k-1} \|(I+\dt \cP)^q  (\cT_i \phi) \|_\infty   ,\qquad \forall k = 1,\ldots, m+1\,. \label{eq:induction_c}
\end{eqnarray}
It is straightforward to prove~\eqref{eq:induction_c} by induction. For $0\leq q \leq m+1$, we have
\[
e_{(I+\dt\cP)^q \phi}^0 = \bbE^0 [( (I+\dt\cP)^q  \phi) (x^0,v^0)] - \int f(0,x,v)\, ((I+\dt\cP)^q \phi) (x,v) \rd x \rd v = 0\,.
\]
Moreover, when $\dt$ is small, 
% \begin{eqnarray*}
%   \sum_{q=0}^{k-1} C_k^{q+1} \dt^{q} \cP^{q}   (\cT_i \phi) \leq     k \sum_{q=0}^{k-1} C_{k-1}^{q} \dt^{q} \cP^{q}   (\cT_i \phi) \approx k  (I + \dt \cP)^{k-1} (\cT_i \phi) \approx k \,\psi_i(t=t^{k-1},x,v)\,,
% \end{eqnarray*}
\[
(I+\dt \cP)^q  (\cT_i \phi) \approx \psi_i(t = q\dt)\,,
\]
where $\psi_i$ solves
\begin{equation}\label{eq:psi eqns}
    \begin{cases}
    \partial_t \psi_i (t,x,v) &= \cP \psi (x,v), \\
    \psi_i(t=0,x,v ) &= \cT_i \phi(x,v),
    \end{cases}
\end{equation}
for $i = 1,2,3$. Note that~\eqref{eq:psi eqns} is the same as~\eqref{eq:reverse RTE}, which is the regular RTE with the velocity sign reversed. Thus, \eqref{eq:psi eqns} enjoys all the properties of the forward RTE~\eqref{eqn:RTE}. We then have 
\[
\|(I+\dt \cP)^q  (\cT_i \phi) \|_\infty
\leq \norm{\cT_i \phi}_\infty , \quad \forall q \geq 0\,,
\]
and hence
\[
\dt^2  \sum_{i=1}^3 \mathfrak{c}_i \sum_{q = 0}^{k-1} \|(I+\dt \cP)^q  (\cT_i \phi) \|_\infty  \leq k \dt^2 \sum_{i=1}^3 \mathfrak{c}_i \|\cT_i \phi \|_\infty\,.
\]
Therefore, using~\eqref{eq:induction_c} for $k=m+1$, we have
\[
e_\phi^{m+1} \leq  (m+1)  \dt^2 \sum_{i=1}^3 \mathfrak{c}_i \|\cT_i \phi\|_\infty   \leq  (m+1) C_1 \dt^2\,,
\]
where $C_1 $ is defined in~\eqref{C1}. The same analysis applies to $e_{-\phi}^{m+1}$ with precisely the same upper-bound $(m+1) C_1 \dt^2$. Since $e_{-\phi}^{m+1} = - e_{\phi}^{m+1}$, we have $|e_\phi^{m+1}| < (m+1) C_1 \dt^2$.
\end{proof}

% If we choose $C_2$ based on~\eqref{C1}, then~\eqref{eq:induction0} becomes
% \begin{align*} %\label{14}
%     \big|e_\phi^{m+1}\big| \leq  (1+ C_2 \Delta t) \sup_{\phi \in \Phi} \big|e_\phi^m\big|  + C_1 \Delta t^2\,,
% \end{align*}
% and \eqref{eq:induction1} directly follows. 

% Note that at the initial time, we have $e_\phi^0 = \bbE^0 [\phi(x^0,v^0)] - \int f(0,x,v) \phi(x,v) \rd x \rd v = 0$. Using~\eqref{eq:induction0} through induction, we have that for any $m = 1,2,\ldots,M$,
% \begin{align*}
%   \big| e_\phi^m \big| & \leq (1+C_2\Delta t)^{m-1} C_1 \Delta t^2 + C_1 \Delta t^2 \sum_{j=0}^{m-2} (1+C_2 \dt
%     )^j = C_1 \Delta t^2 \sum_{j=0}^{m-1} (1+C_2 \dt )^j\,,
% \end{align*}
% concluding the proof.
% \end{proof}

Inductively, \Cref{lem:expectation} allows us to arrive at the following \Cref{prop:expectation}. It is essentially the first bullet point of~\Cref{thm:f_conv_summary} and we spell out the constant dependence.
\begin{proposition}\label{prop:expectation}
Given a fixed time $T\geq 0$,   $M \in \mathbb{N}$, and $\dt =  T/M$,  for any $\phi\in\Phi$, we have $\big| e_\phi^M\big|  \leq C_1 \dt $, 
where $e^M_\phi$ is defined in~\eqref{eqn:def_error_exp} and $C_1$ is defined in \eqref{C1}.
% \begin{equation} \label{C3}
% C_3 = C_3(\norm{f_\In}_{W^{2,\infty}_x L_v^\infty}, \norm{\sigma}_{W_x^{1,\infty}}, \norm{\phi}_{W_{x,v}^{2,\infty}}, \norm{\cP \phi}_{W_{x,v}^{2,\infty}},\ldots,  \norm{\cP^{M-1}\phi}_{W_{x,v}^{2,\infty}} , T) \,.
% \end{equation}
\end{proposition}

\begin{proof}
Recall that $\bbE^M$ includes all the randomness  along the trajectory to obtain the final-time particle $(x^M, v^M)$. Based on~\Cref{lem:expectation}, we obtain an upper bound for $e_\phi^M$ by setting $m=M$ and $T = M \dt $.
\end{proof}
We should note that though~\Cref{prop:expectation} is formulated for a one-particle system, due to the fact that $\mathsf{e}_{N,\phi}^M$ is the simple $N$-average of $\mathsf{e}^M_{1,\phi}$, the mean is preserved, and the extension to the many-particle-system is trivial. This concludes~\eqref{eq:f_convergence_exp_summary} in~\Cref{thm:f_conv_summary}.

\subsubsection{Many Particles}\label{sec:many_particles}
Using formula~\eqref{eqn:N_error}, we view $\mathsf{e}_{N,\phi}^m$ as the average of many samples drawn i.i.d.~according to the distribution of $\mathsf{e}^m_{1,\phi}$. Thus, we can directly apply the concentration tail bound. To do so, we first cite the famous Bernstein inequality.
\begin{theorem}[Bernstein Inequality]
Let $x_1,\ldots,x_N$ be i.i.d.~real-valued samples of a random variable $X$, whose expectation and variance are $\mu = \mathbb{E}[X]$ and $\sigma^2 = \var[X]$. Assume that there exists $b$ such that $|X -\mu| \leq b$ almost surely. Then for any $t >0$, we have
\begin{equation}\label{eqn:bernstein}
\mathbb{P}\left(\frac{1}{N}\sum_{n=1}^N x_n - \mu \geq t \right) \leq \exp \left( - \frac{Nt^2}{2\left(\sigma^2 + \frac{1}{3}bt \right)} \right)\,.
\end{equation}
\end{theorem}
According to~\eqref{eqn:bernstein}, to control the tail of $\mathsf{e}^m_{N,\phi}$, we need to give an estimate on the variance ($\sigma^2$ above), and the range ($b$ above). These bounds are presented in the following lemma.

\begin{lemma}\label{lem:bound_domain_var}
Let $\mathsf{e}^m_{1,\phi}$ be defined in~\eqref{eqn:error_phi} with $\mathsf{f}^m_1$ computed through~\Cref{alg:f-RTE}. Then $\mathsf{e}^{m}_{1,\phi}$ has a bounded range, and bounded variance. In particular, we have, for all $m \leq M$ and $\phi \in \Phi$:
\begin{itemize}
    \item The range of the error term is bounded: 
\begin{equation}\label{eqn:error_domain_bound}
|\mathsf{e}^m_{1,\phi}| \leq 2 \norm{\phi}_{L_{x,v}^\infty} \,.
\end{equation}
\item The variance is bounded
\begin{equation}\label{eqn:error_var_bound}
\var[\mathsf{e}^m_{1,\phi}]\leq C_2\,,
\end{equation}
where $C_2$ depends on $\|v\|_{L^\infty(\Omega)}$, $\norm{\sigma}_{L^\infty}$, $T$ and $\norm{\phi}_{W_{x,v}^{1,\infty}}$.
\end{itemize}
\end{lemma}
\begin{proof}
To show~\eqref{eqn:error_domain_bound}, recall
\[
    \nume_{1,\phi}^m = \phi(x^m, v^m) - \int f(t^m) \phi(x,v) \rd x \rd v\,,
\]
then we have 
\[
    \big|\nume_{1,\phi}^m\big| \leq \big|\phi(x^m, v^m)\big| + \bigg|\int f(t^m) \phi(x,v) \rd x \rd v \bigg|
 \leq \norm{\phi}_{L_{x,v}^\infty}   + \norm{\phi}_{L_{x,v}^\infty} \int f(t^m, x, v) \rd x \rd v 
    = 2 \norm{\phi}_{L_{x,v}^\infty}\,,
\]
where we have used the fact that $f(t^m)$ is non-negative and integrates to one thanks to the mass conservation. 

Since we have just proved~\eqref{eqn:error_domain_bound} which indicates $|\bbE[\nume_{1,\phi}^m]|^2  \leq 4 \norm{\phi}_{L_{x,v}^\infty}^2$,  we will show the boundedness of the second-order moment $V_\phi^{m} :=\bbE[|\mathsf{e}^m_{1,\phi}|^2] $ to demonstrate that the variance is bounded. Note that
\[
V_\phi^{m}  = \bbE \left[|\langle\mathsf{f}^m_1-f(t^m)\,,\phi\rangle_{xv}|^2 \right]=\bbE[| \phi(x^m,v^m)-\langle f(t^m)\,,\phi\rangle_{xv} |^2]\,.
\]
Then the proof follows from induction. We begin with the following estimate:
\begin{equation}\label{eqn:variance_0}
\begin{aligned}
V_\phi^{0}&=\mathbb{E}^0\left[ |\phi(x^0,v^0)-\langle f_\In\,,\phi\rangle_{xv}|^2 \right] = \mathbb{E}^0\left[ \phi(x^0,v^0)^2 - 2 \phi(x^0,v^0) \langle f_\In,\phi\rangle_{xv} + \langle f_\In,\phi\rangle^2_{xv} \right]\\
&\leq \norm{\phi}_{L_{x,v}^\infty}^2  + 2 \norm{\phi}_{L_{x,v}^\infty}^2 \int f_\In \rd x \rd v + \norm{\phi}_{L_{x,v}^\infty}^2 \left( \int f_\In \rd x \rd v\right)^2 = 4  \norm{\phi}_{L_{x,v}^\infty}^2 \,,
\end{aligned}
\end{equation}
where we use the non-negativity of $f_\In$ and that its mass equals one. We also show below that $V_\phi^m$ does not grow fast in $m$. For any $m\geq 0$\,,
% \end{proof}
% \begin{lemma}\label{lem:variance}
% For final time $T\geq 0$, and let $\dt$ be the time discretization and let $M = T/\dt$. Let $(x^m,v^m)$ be one particle generated by Algorithm~\ref{alg:f-RTE}, then for all $\phi \in \Phi$, $\var[\phi(x^M, v^M)]  \leq  C_3 $, where $C_3$ is a constant defined in \eqref{C3}, for all $n\leq M$.
% \end{lemma}
% \begin{proof}
% Using all the earlier notations,  we write 
% \begin{align*}
%     \var[\phi(x^{m+1}, v^{m+1})]  
%     & = \bbE^\np \left[ \left( \phi(x^\np, v^\np) - \bbE^\np \phi(x^\np, v^\np) \right)^2 \right]
%     \\ & = \bbE^\np \left[ \left( \phi(x^\np, v^\np) - \average{f(t_\np), \phi}_{xv} - e_\phi^{m+1}\right)^2 \right]
% \\& \leq  V_\phi^{m+1} + C_1((m+1)^2 \dt^4)\,,
% \end{align*}
% \ql{last equation misses $+2\sqrt{V^{n+1}_\phi e^{n+1}_\phi}$?} where the last inequality uses a variant of Theorem~\ref{prop:expectation}, and 
\begin{align}
 &V_\phi^{m+1}\nonumber\\
 =& \bbE^{m+1} \left[ |\phi(x^{m+1}, v^{m+1}) - \langle f(t^{m+1}),\phi\rangle_{xv} |^2 \right]  \nonumber \\
 =& \bbE^{m+1} \Big[  | \underbrace{\phi(x^{m}, v^{m}) - \langle f(t^m),\phi\rangle_{xv}}_{A_1} + \underbrace{\phi(x^{m+1}, v^{m+1}) - \phi(x^{m}, v^{m})   +\langle f(t^m)-f(t^{m+1}),\phi\rangle_{xv}}_{A_2} |^2 \Big]\nonumber\\
 = &  \bbE^{m+1} \left[ |A_1|^2 \right] + \bbE^{m+1} \left[ |A_2|^2 \right] + 2\bbE^{m+1} \left[ A_1 A_2 \right] \nonumber\\
 = & V_\phi^{m} +  \bbE^{m}\left[  \bbE_{\eta^{m+1}}\bbE_{p^{m+1}} \left[ |A_2|^2 \right]\right] + 2 \bbE^{m} \left[ A_1  \bbE_{\eta^{m+1}}\bbE_{p^{m+1}}[A_2] \right]  \, .
 \label{eq:var_exp}
\end{align}
First, using the mean-value theorem, we see that  
\begin{align*}
  \bbE_{\eta^\np}\bbE_{p^\np}[A_2]  \leq  & \dt \|v\cdot \nabla_x \phi + \sigma \cL[\phi]\|_\infty \|f(t=\xi,x,v)\|_1  + \left(1- e^{-\sigma(x^{m+1})\Delta t} \right) \cL[\phi] (x^{m+1})  \nonumber\\
 &  + \dt \, \nabla_x \phi(\xi_x,v^m)\cdot v^m\,,\nonumber \\
 \leq  &  \dt \,  C_2(\|v\|_{L^\infty(\Omega)}, \norm{\sigma}_{L^\infty} , 
\norm{\phi}_{W_{x,v}^{1,\infty}})\,,
\end{align*}
where $ t^m\leq \xi \leq t^\np$ and $ \xi_x$ is between $x^m$ and $x^\np$.
Thus, 
$$
2 \bbE^{m} \left[ A_1  \bbE_{\eta^{m+1}}\bbE_{p^{m+1}}[A_2] \right]  \leq \dt C_2 \, \bbE^{m} \left[ A_1\right] \leq C_2 \dt\,,
$$
where we have used~\eqref{eqn:error_domain_bound}. Next, we focus on $\bbE_{\eta^{m+1}}\bbE_{p^{m+1}} \left[ |A_2|^2 \right]$.
\begin{align*}
\bbE_{\eta^{m+1}}\bbE_{p^{m+1}} \left[ |A_2|^2 \right] =   & \left(1- e^{-\sigma(x^{m+1})\Delta t} \right) \Big| |\Omega|^{-1} \langle \phi \rangle_v (x^\np) - \phi(x^m,v^m) +  \dt \langle \cP\phi, f(t=\xi) \rangle_{xv} \Big|^2\\
   & + \dt^2 e^{-\sigma(x^{m+1})\Delta t}  \Big| \nabla_x \phi(\xi_x,v^m)\cdot v^m + \langle \cP\phi, f(t=\xi)  \rangle_{xv}  \Big|^2 \\
   \leq & \dt \,  C_2(\|v\|_{L^\infty(\Omega)}, \norm{\sigma}_{L^\infty},  
\norm{\phi}_{W_{x,v}^{1,\infty}})\,. 
\end{align*}
As a result,
\begin{equation} \label{C2}
\bbE^{m+1} \left[ |A_2|^2 \right] \leq \dt \, C_2(\|v\|_{L^\infty(\Omega)}, \norm{\sigma}_{L^\infty},  
\norm{\phi}_{W_{x,v}^{1,\infty}}) \,.
\end{equation}
Combining the above two inequalities together, we have
\begin{equation}\label{eqn:variance_growth}
     V_\phi^{m+1}  \leq  V_\phi^{m} + C_2 \,\dt\,.
\end{equation}
Note that $\var[\mathsf{e}^m_{1,\phi}] = V_\phi^m - |\mathbb{E}[\mathsf{e}^m_{1,\phi}]|^2$,~\eqref{eqn:error_var_bound} is obtained by using the boundedness of both terms from~\eqref{eqn:error_domain_bound} and~\eqref{eqn:variance_growth}, with the constant depending on $T$ and the arguments listed on~\eqref{C2}. We abuse the notation and still call it $C_2$.
\end{proof}

We are now ready to show~\eqref{eq:f_convergence_summary} in~\Cref{thm:f_conv_summary} through the following proposition.
\begin{proposition}\label{prop:f_convergence}
Let $\numf^m_N$ be the solution from running~\Cref{alg:f-RTE} with MC samples $\{(x^m_n, v^m_n)\}_{m=0}^M$. We claim it approximates $f$ with high probability. Namely, for small enough $\epsilon,\dt \ll 1$, and any $\phi \in \Phi$, the chance of the weak error~\eqref{eqn:error_phi} being bigger than $\mathcal{O}(\epsilon+\dt)$ is exponentially small in $N$:
\[ %\label{eq:f_convergence}
    \mathbb{P}\left(|\mathsf{e}^m_{N,\phi}| \geq \epsilon + C_1 \dt \right) \leq 2 \exp \left(-   \frac{N \epsilon^2}{3C_1}\right)\,.
\]
\end{proposition}
\begin{proof}
Calling~\Cref{prop:expectation} and~\Cref{lem:bound_domain_var}, we have shown that for any $\phi \in \Phi$, 
\begin{eqnarray*}
\bbE\left[ |\mathsf{e}^m_{1,\phi}|\right]  \leq C_1 \dt,\quad 
\var\left[\mathsf{e}^m_{1,\phi}\right] \leq C_2\,,\quad\text{and}\quad |\mathsf{e}^m_{1,\phi}|< 2  \norm{\phi}_{L_{x,v}^\infty}\,.
\end{eqnarray*}
%\ql{I added the last $C_3$ above, not sure.}
Noting~\eqref{eqn:N_error}, we apply the Bernstein inequality~\eqref{eqn:bernstein} to have, for any $\epsilon > 0 $:
\begin{eqnarray*}
%  \mathbb{P}\left( |\mathsf{e}^m_{N,\phi}| \geq  \epsilon + C_3 \dt \right) \leq \mathbb{P}\left( | \mathsf{e}^m_{N,\phi} - {e}^m_{1,\phi} |   \geq \epsilon  \right) \leq  2 \exp  \left( - \frac{N \epsilon^2}{2 C_3 +  2\epsilon   (2L_\Phi + C_3 \dt)/3}\right)\,.
 \mathbb{P}\left( |\mathsf{e}^m_{N,\phi}| \geq  \epsilon + C_1 \dt \right) \leq \mathbb{P}\left( | \mathsf{e}^m_{N,\phi} - {e}^m_{1,\phi} |   \geq \epsilon  \right) \leq  2 \exp  \left( - \frac{N \epsilon^2}{2 C_2 +  2\epsilon   (2 \norm{\phi}_{L_{x,v}^\infty} + C_1 \dt)/3}\right)\,.
\end{eqnarray*}
We conclude the proof by absorbing the constants into $C_1$ and setting $\epsilon,\dt$ to be sufficiently small.
\end{proof}
We have completed the proof of~\Cref{thm:f_conv_summary} with the constants' dependence explicitly spelled out.

\begin{remark}
In the next section, we will directly show that the computation of the Fr\'echet derivative is accurate. The proof for the solver for $g$, as expressed in~\eqref{gi}, is not directly needed, and we nevertheless make some comments here. As stated in~\Cref{rmk:g_solver}, the derivation for the algorithm only comes from the conservation constraint~\eqref{tfg}, but this constraint cannot fully represent the entire dynamical information of the adjoint equation~\eqref{eqn:g_eqn}. To show that the solver is consistent with the adjoint solver, we need to return to the equation and test $\mathsf{g}-g$ on a given smooth function $\phi$ and trace the evolution of the error. This comes down to proving that the error is not enlarged by much in every iteration:
\[
\langle\mathsf{g}^{m+1}_1-g(t^{m+1})\,,\phi\rangle_{xv}- \langle\mathsf{g}^{m}_1-g(t^{m})\,,\phi\rangle_{xv}\sim 0\,.
\]
While $\langle\mathsf{g}_1^{m+1}-\mathsf{g}_1^{m}\,,\phi\rangle_{xv}=\phi(x^{m+1},v^{m+1})-\phi(x^{m},v^{m})$ which translates to the algorithmic relation between $(x^{m+1},v^{m+1})$ and $(x^{m},v^{m})$, the term $\langle g(t^{m+1})-g(t^{m})\,,\phi\rangle$ calls for the Taylor expansion in time for the PDE. Following the same strategy as shown in~\Cref{lem:expectation}, one can show that the same trajectory $\{(x^m,v^m)\}$ represents the $g$ dynamics backward in time. We should also note that unlike the MC solver for $f$ where each particle is i.i.d.~sampled, initiated from the initial distribution to represent $f_\In$, this solver for $g$ encodes a non-uniform weight for each particle. The particle takes on the value for $g(T)$. This difference is also recognized in different particle methods; see~\cite{raviart1985analysis}.
\end{remark}
% \lw{can we prove that \eqref{gi} solves the adjoint equation?}
% \yy{Following the deterministic particle method, it is possible, but the interpolation error is intractable. Our approximation to the function $g$ is bad if the $f$ trajectory hasn't visited there} \lw{Then I am fine with the above remark.}

\subsection{Convergence of the Gradient}
This section is dedicated to~\Cref{thm:gradient_summary}, which shows that  \Cref{alg:P-OTD} provides an accurate numerical approximation to the true numerical gradient~\eqref{eqn:frechet_derivative_semi} in its semi-discrete-in-time form. The proof is conducted on the weak formulation to eliminate the complication from the delta function. Namely, we will examine the difference between $\mathfrak{G}_{\phi}$ and $\mathfrak{G}_{N,\phi}$, separately defined in~\eqref{eqn:semi-grad} and~\eqref{eqn:num-grad}. Since they both consists of two terms, we show below that $\mathfrak{G}^m_{N,i,\phi}$ defined in~\eqref{eqn:frakG_dis} approximates $\mathfrak{G}^m_{i,\phi}$ defined in~\eqref{eqn:frakG_cont} for $i=1,2$, and for all $1\leq  m\leq M$. The results are encapsulated in the following two propositions.

\begin{proposition}\label{prop:frakG_1}
If the initial condition $f_\In$ and measurement $d(x,v)$ are sufficiently regular such that they both are $C_c^\infty(\RR^{d_x}\times \Omega)$, 
%\ql{I am not quite comfortable that at this stage, we still say it is sufficiently regular... the paper is more or less done, so things should probably be less formal?}such that $\phi(x) g(t^m,x,v) \in \Phi$, 
then for small $\epsilon$ and $\dt$,
\[ 
    \mathbb{P}\left(| \frakG_{N,1,\phi}^{m}-  \frakG_{1,\phi}^m | \geq \epsilon + C_3 \dt \right) \leq 2 \exp \left(-   \frac{N \epsilon^2}{3C_3}\right), \qquad 1\leq m \leq M\,,
\]
where 
\begin{align} \label{C5}
    C_3 = C_3(\norm{f_\In}_{W^{2,\infty}_x L_v^\infty}, \norm{\sigma}_{W_x^{1,\infty}}, \norm{\phi}_{W_{x,v}^{2,\infty}}, \norm{d}_{W_{x}^{2,\infty}L_v^\infty},  T)\,.
\end{align}
\end{proposition}

\begin{proposition}\label{prop:frakG_2}
Under the same assumptions with~\Cref{prop:frakG_1}, we have 
%\yy{I change the following results statement as well since they all share the same assumptions} \lw{looks good}
% If the initial condition $f_\In$ and measurement $d(x,v)$ are sufficiently regular such that $\phi(x) \average{g(t^m,x, \cdot)}_v \in \Phi$, then for small $\epsilon$ and $\dt$, 
\[ 
    \mathbb{P}\left(| \frakG_{N,2,\phi}^{m}-  \frakG_{2,\phi}^m | \geq \epsilon + C_3 \dt + |e_v| \right) \leq 2 \exp \left(-   \frac{N \epsilon^2}{3C_3}\right), \qquad  1\leq m \leq M\,,
\]
where $e_v$ is the quadrature error in computing $\average{\numg_n}_v$.
\end{proposition}

\Cref{thm:gradient_summary} is then a direct corollary of the two propositions.
% \begin{proof}[Proof for~\Cref{thm:gradient_summary}]
% \ql{blahblah}
% \end{proof}

We now present the proofs for~\Cref{prop:frakG_1,prop:frakG_2}. Note that the two terms $\frakG_{N,1,\phi}^{m}$ (resp.~$\frakG_{N,2,\phi}^{m}$) and $\frakG_{1,\phi}^m$ (resp.~$\frakG_{2,\phi}^m$) share the same format, so in the following, we will only present details for $\frakG_{N,1,\phi}^{m}$ and $\frakG_{1,\phi}^m$. Consider an auxiliary formulation
\begin{align} \label{0732}
    \tilde{\frakG}_{N,1,\phi}^{m} = \frac{1}{N} \sum_{n=1}^N \phi(x_n^m) g(t^m, x_n^m, v_n^m)\,.
\end{align}
Then by the triangle inequality:
\begin{equation}\label{eqn:frakG_1_12}
  |{\frakG}_{N,1,\phi}^{m}-  {\frakG}_{1,\phi}^{m}|\leq 
   |\frakG_{N,1,\phi}^{m}-  \tilde{\frakG}_{N,1,\phi}^{m}| +  |\tilde{\frakG}_{N,1,\phi}^{m}-  \frakG_{1,\phi}^m|\,.
\end{equation}
\Cref{prop:frakG_1} is a direct corollary of the following~\Cref{lem:G11,lem:G12} that give a control over the two terms on the right-hand side of \eqref{eqn:frakG_1_12}, respectively.

The second term $\tilde{\frakG}_{N,1,\phi}^{m}-  \frakG_{1,\phi}^m$ can be easily controlled since  $\phi(x) g(t^m,x,v) \in \Phi$. We state the following~\Cref{lem:G11} without proof.
\begin{lemma} \label{lem:G11}
%If the initial condition $f_\In$ and measurement $d(x,v)$ are sufficiently regular such that $\phi(x) g(t^m,x,v) \in \Phi$, then for small $\epsilon$ and $\dt$,
Under the same assumptions with~\Cref{prop:frakG_1}, we have
\begin{equation}  \label{G1N}
    \mathbb{P}\left(|\tilde{\frakG}_{N,1,\phi}^{m}-  \frakG_{1,\phi}^m| \geq \epsilon + C_3 \dt \right) \leq 2 \exp \left(-   \frac{N \epsilon^2}{3C_3}\right), \qquad 1\leq m \leq M\,.
\end{equation}
where $C_3$ is defined in~\Cref{prop:frakG_1}.
\end{lemma}

%\yy{Can we change the order for the 1st and 2nd terms?} \lw{I don't have preferences.} 
The first term in~\eqref{eqn:frakG_1_12} reads:
\begin{align} \label{0724}
     \frakG_{N,1,\phi}^{m}-  \tilde{\frakG}_{N,1,\phi}^{m}
     = \frac{1}{N}\sum_{n=1}^N  \phi(x_n^m) \left( \numg_n - g(t^m, x_n^m, v_n^m)\right)\,,
\end{align}
and we control it as follows. 
\begin{lemma} \label{lem:G12}
% If the initial condition $f_\In$ and measurement $d(x,v)$ are sufficiently regular such that $\phi(x) g(t^m,x,v) \in \Phi$, then for small $\epsilon$ and $\dt$,
Under the same assumptions with~\Cref{prop:frakG_1}, we have
\[ 
    \mathbb{P}\left(|\frakG_{N,1,\phi}^{m}- \tilde{\frakG}_{N,1,\phi}^{m}| \geq \epsilon + C_3 \dt \right) \leq 2 \exp \left(-   \frac{N \epsilon^2}{3C_3}\right), \qquad 1\leq m \leq M\,.
\]
\end{lemma}
\begin{proof}
First note that due to~\eqref{tfg}, we always have 
\[ %\label{0722}
\int_{xv} f(t^M,x,v) g(t^M,x,v) \rd x \rd v = \int_{xv} f(t^m,x,v) g(t^m,x,v) \rd x \rd v \,, \qquad 1\leq m \leq M. 
\]
Again, we assume that $g$, as a solution of the adjoint RTE, is sufficiently smooth such that $g(t, \cdot, \cdot ) \in \Phi$. Then, using a variant of~\Cref{prop:expectation} (i.e., following the same proof of~\Cref{prop:expectation} but starting at time $t^m$ instead of $t_0$), we have %and recall the definition of $\bbE^n$ in \eqref{bbEn},
\[ %\label{0807}
\tilde\bbE^m[  g(t^M,x_n^M, v_n^M) |(x_n^m,v_n^m)]= g(t^m, x_n^m, v_n^m) + e\,,
\]
with $|e| \leq C_3 \dt$, and $\tilde\bbE^m$ defined as
\begin{equation}\label{eq:expectation_n_2}
    \tilde\bbE^m[~\cdot~|(x_n^m,v_n^m)] = \bbE_{\eta^\np} \bbE_{p^\np} \cdots \bbE_{\eta^M} \bbE_{p^M}[~\cdot~ | (x_n^m, v_n^m)] \,.
\end{equation}
Recall the definition of $\numg_n$ in \eqref{gi}. We then have
\[ 
    \tilde \bbE^m\left[\numg_n \right]= g(t^m, x_n^m, v_n^m) + e\,.
\]
With the same argument, we have that
\begin{equation}\label{0731}
     \tilde \bbE^m \left[\phi(x_n^m) \numg_n  \right]= \phi(x_n^m) g(t^m, x_n^m, v_n^m) + e\, \phi(x_n^m)\,. 
\end{equation}
Therefore, \eqref{0724} becomes
\begin{align*}
     \frakG_{N,1,\phi}^{m}-  \tilde{\frakG}_{N,1,\phi}^{m}
      = \frac{1}{N}\sum_{n=1}^N   \left(\phi(x_n^m)\numg_n - \tilde\bbE^m \left[\phi(x_n^m) \numg_n \right]+ e \, \phi(x_n^m) \right), \quad |e\, \phi(x_n^m)| \leq C_3 \Delta t\,.
\end{align*}
Note further from~\Cref{lem:bound_domain_var} (here we again alter the Lemma by considering the randomness starting at time $t^\np$, and the result remains the same), we have
\begin{align*}
    \var[\phi(x_n^m)\numg_n]&= \tilde\bbE^m \left[ \phi(x_n^m)\numg_n - \tilde \bbE^m \left[\phi(x_n^m) \numg_n \right]\right]^2  = \phi(x_n^m)^2 \tilde\bbE^m \left[ \numg_n - \tilde \bbE^m  \left[\numg_n \right]\right]^2  \leq C_3\,.
\end{align*}
Then the final result follows from the Bernstein inequality~\eqref{eqn:bernstein}.
\end{proof}

% Combining Lemma~\ref{lem:G11} and Lemma~\ref{lem:G12}, we have 
% \begin{proposition}\label{prop:frakG_1}
% If the initial condition $f_\In$ and measurement $d(x,v)$ are sufficiently regular such that $\phi(x) g(t^m,x,v) \in \Phi$, then for small $\epsilon$ and $\dt$, 
% \begin{equation} 
%     \mathbb{P}\left(| \frakG_1^{N,n}-  \frakG_{1,\phi}^m | \geq \epsilon + C_3 \dt \right) \leq 2 \exp \left(-   \frac{N \epsilon^2}{3C_3}\right), \qquad 1\leq n \leq M\,.
% \end{equation}
% \end{proposition}

% The same result holds for $\frakG_2^{N,n}-  \frakG_{2,\phi}^m$ except that we need to take into an additional error coming from the numerical quadrature in $v$. More precisely, we have the following theorem.
% \begin{proposition}\label{prop:frakG_2}
% If the initial condition $f_\In$ and measurement $d(x,v)$ are sufficiently regular such that $\phi(x) \average{g(t^m,x, \cdot)}_v \in \Phi$, then for small $\epsilon$ and $\dt$, 
% \[ 
%     \mathbb{P}\left(| \frakG_2^{N,n}-  \frakG_{2,\phi}^m | \geq \epsilon + C_3 \dt + |e_v| \right) \leq 2 \exp \left(-   \frac{N \epsilon^2}{3C_3}\right), \qquad  \quad 1\leq n \leq M\,,
% \]
% where $e_v$ is the quadrature error in computing $\average{\numg_i}_v$.
% \end{proposition}
The proof for~\Cref{prop:frakG_2} is almost identical. We present the details below.
\begin{proof}[Proof of~\Cref{prop:frakG_2}]
Parallel to \eqref{0732}, we define the auxiliary function 
\[
    \tilde{\frakG}_{N,2,\phi}^{m} = \frac{1}{N} \sum_{n=1}^N \phi(x_n^m) \average{g(t^m, x_n^m, \cdot)}_v\,,
\]
and thereby write 
\[
  \frakG_{N,2,\phi}^{m}-  \frakG_{2,\phi}^m = 
   \frakG_{N,2,\phi}^{m}-  \tilde{\frakG}_{N,2,\phi}^{m} +  \tilde{\frakG}_{N,2,\phi}^{m}-  \frakG_{2,\phi}^m\,.
\]
Here, the term $\frakG_{N,2,\phi}^{m}-  \tilde{\frakG}_{N,2,\phi}^{m}$ has exactly the same estimate as in \eqref{G1N}. For $\tilde{\frakG}_{N,2,\phi}^{m}-  \frakG_{2,\phi}^m$, in addition to the error between $\numg_n$ and $g(t^m, x_n^m, v_n^m)$, there is another error when computing the integral in $v$. 
%\yy{I think the following is confusing because the integral has to be time-dependent. It does not reflect time-dependence in its current notations.} \lw{I add the time dependence} 
More precisely, if $\numg_n^m = g(t^m, x_n, v_n)$, then   $\average{\numg^m}_v(\bar x)$ obtained via quadrature formula \eqref{eqn:gn}, which approximates the real integral $\int g(t^m, \bar x, v) \rd v$ with a quadrature error denoted as $e_v$.  Now, we follow the same proof as in~\Cref{lem:G12} until at~\Cref{0731}, where we need to include this additional quadrature error to get  
% \begin{align*}
%      \tilde \bbE^n \left[ \average{\numg_i}_v \right]= \average{ g(t^m, x_n^m, \cdot)}_v + C_3 \dt + e_v\,.
% \end{align*}
% Here $e_v$ is the quadrature error in $v$ when using $\average{\cdot}_v$ to approximate the integral $\average{\cdot}_v$. Then, 
\[
     \tilde \bbE^m \left[ \phi(x_n^m) \average{\numg_n}_v \right] =\phi(x_n^m) \average{ g(t^m, x_n^m, \cdot)}_v + e\, \phi(x_n^m) + e_v \,.
\]
Therefore,
\begin{align*}
     \frakG_{N,2,\phi}^{m}-  \tilde{\frakG}_{N,2,\phi}^{m}
      = \frac{1}{N}\sum_{n=1}^N   \left( \phi(x_n^m)\average{\numg_n}_v 
      - \tilde\bbE^m \left[\phi(x_n^m) \average{\numg_n}_v \right]+ e\, \phi(x_n^m) + e_v \right).
\end{align*}
Same arguments as in~\Cref{lem:G12}
show that
$\var[\phi(x_n^m) \average{\numg_n}_v] \leq C_3$, and the boundedness of $\phi(x_n^m) \average{\numg_n}_v$ is a direct consequence of the fact that both $\phi$ and  $\numg$ are bounded. 
Then by the Bernstein inequality, we have 
\begin{eqnarray*}
\mathbb{P}\left( | \frakG_{N,2,\phi}^{m}-  \tilde{\frakG}_{N,2,\phi}^{m}-C_3\Delta t - e_v |   \geq \epsilon  \right) \leq  2 \exp  \left( - \frac{N \epsilon^2}{2 C_3 +  2\epsilon   (2 \norm{\phi}_{L_{x,v}^\infty}+ C_3 \dt)/3}\right).
\end{eqnarray*}
Then if $\epsilon$ and $\dt$ are small enough such that $2\epsilon   (2\norm{\phi}_{L_{x,v}^\infty}+ C_3 \dt)/3 < C_3$, we have
\begin{eqnarray*}
 \mathbb{P}\left( |\frakG_{N,2,\phi}^{m}-  \tilde{\frakG}_{N,2,\phi}^{m}| \geq  \epsilon + C_3  \dt + |e_v| \right) 
 \leq \mathbb{P}\left( | \frakG_{N,2,\phi}^{m}-  \tilde{\frakG}_{N,2,\phi}^{m}-C_3\Delta t - e_v |   \geq \epsilon  \right)
 \leq  2 \exp  \left( - \frac{N \epsilon^2}{3C_3}\right).
\end{eqnarray*}
% \ql{we did not prove the boundedness of the random variable. Would that be okay?}
% \lw{I add one sentence before applying Bernstein. }
\end{proof}

\section{Discretize-Then-Optimize Framework}\label{sec:DTO}
In this section, we consider the so-called Discretize-Then-Optimize (DTO) framework to compute the gradient of an RTE-constrained optimization problem. We regard the Monte Carlo method in~\Cref{alg:f-RTE} as our discretization of the forward RTE~\eqref{eqn:RTE}. 

Consider the objective functional
\begin{equation}\label{eq:J2}
    J = \iint r(x,v) f(T,x,v) \rd x \rd v\,,
\end{equation}
where the mean-field quantity $J$ of the final-time RTE solution $f(T,x,v)$ is our measured data. The choice of $J$ here is only for the convenience of derivation, and the same method proposed in this section shall apply to general objective functionals. To carry on with the notation in \Cref{subsec:fn}, the value in \eqref{eq:J2} could be approximated by the Monte Carlo quadrature
\begin{equation}\label{eq:J2_MC}
    J \approx \frac{1}{N} \sum_{n=1}^N r(x_n^M, v_n^M) =: \mathcal J  \,.
\end{equation}
Since we have already established the convergence theory of the Monte Carlo method for RTE in~\Cref{subsec:fn}, here we assume  that~\eqref{eq:f_n_approx} holds and
$$
(x_n^m, v_n^m) \sim f(t^m,x,v),\quad \forall n=1,\ldots,N, \quad m= 0,\ldots M.
$$
% up to the random and time discretization errors. 

%We denote by $\mathbb{E}_i^{k+1}$ the expectation over all the randomness in sampling $v_n^{k+1}$ given $(x_n^k, v_n^k)$ (while $x_n^{k+1} = x_n^k + \Delta t v_n^k$ is deterministic); see~\Cref{alg:f-RTE} for details. 
Given any test function $\phi(v)$ in the admissible set~\eqref{eq:test_fcts}, we have
\begin{equation} \label{eq:probab}
    \bbE_{\eta^{m+1}} \bbE_{p^{m+1}} [\phi(v) | (x_n^m,v_n^m) ] = \alpha_n^{m+1}  \phi(v_n^m) + (1-\alpha_n^{m+1} ) \frac{1}{|\Omega|}\int_\Omega \phi(\eta)d\eta, \quad m = 0,\ldots, M-1,
\end{equation}
where $\alpha_n^{m+1} = \exp(-\sigma(x_n^{m+1}) \Delta t) = \exp(-\sigma(x_n^{m} + \Delta t v_n^m) \Delta t) $. As a result, the final objective functional can be expressed as sums of conditional  expectations. We refer the reader to~\cite{yang2022adjoint} for more details regarding this technique. Next, we use the $\tilde{\bbE}^m$ notation from~\eqref{eq:expectation_n_2} for $m=0,\ldots,M-1$, and~\eqref{eq:J2} can be expressed as
\begin{equation}\label{eq:J_conditional}
    J =  \frac{1}{N} \sum_{n=1}^N \mathbb{E}_{(x_n^{m},v_n^{m})\sim f(t^m,x,v ) } \left[ \tilde{\bbE}^{m} [r(x_n^M,v_n^M)  ]  \right], \quad 0\leq m \leq M-1,
\end{equation}
as a result of the law of total expectations. Note that the above formula holds for any $m$.
% \begin{eqnarray}
% J &=&  \frac{1}{N} \sum_{i=1}^N \mathbb{E}_{(x_n^{M-1},v_n^{M-1})\sim f(t_{M-1},x,v ) } \left[ \tilde{\bbE}^{M-1} [r(x_n^M,v_n^M)  ]  \right] \nonumber
% \\
% &=& \frac{1}{N} \sum_{i=1}^N    \mathbb{E}_{(x_n^{M-2},v_n^{M-2})\sim f(t_{M-2},x,v) }   \left[ \mathbb{E}_i^{M-1} \mathbb{E}_i^{M} [r(x_n^M,v_n^M)  \big |  (x_n^{M-2}, v_n^{M-2} ) ] \right]   \nonumber
% \\
% &=& \frac{1}{N} \sum_{i=1}^N  \mathbb{E}_{(x_n^{k},v_n^{k})\sim f(t_{k},x,v ) }  \left[   \mathbb{E}_i^{k+1} \ldots \mathbb{E}_i^{M} [r(x_n^M,v_n^M)  \big |  (x_n^{k}, v_n^{k} ) ] \right] \nonumber
% \\
% &=&\frac{1}{N} \sum_{i=1}^N  \mathbb{E}_{(x_n^{0},v_n^{0})\sim f(0,x,v)}  \left[ \mathbb{E}_i^{1} \mathbb{E}_i^{2} \ldots \mathbb{E}_i^{M} [r(x_n^M,v_n^M) \big |  (x_n^{0} ,v_n^{0}) ]  \right]. \label{eq:J_conditional}
% \end{eqnarray}
% The above equations hold as a result of the law of total expectations.
We can further write $J$ into $J = \sum_{n=1}^N J_n$ following~\eqref{eq:J_conditional}, where
\begin{eqnarray}
 J_n &=&  \frac{1}{N} \mathbb{E}_{(x_n^{m},v_n^{m})\sim f(t^m,x,v ) } \left[ \mathcal{R}_n^m(x_n^m, v_n^m)  \right] \, , \label{eq:Ji_conditional} \\
%  \mathcal{R}_i^n &=&  \mathbb{E}_i^{n+1} \ldots \mathbb{E}_i^{M} \left[ r(x_n^M,v_n^M)\big |  (x_n^{m}, v_n^{n} ) \right].  
\mathcal{R}_n^m(x_n^m, v_n^m)  &=&  \tilde{\bbE}^{m} [r(x_n^M,v_n^M)  ] \, . \label{eq:Ri_def}
\end{eqnarray}

The dependence of $\mathcal{R}_n^m(x_n^m, v_n^m)$ on the coefficient function $\sigma(x)$ is through the evaluations of $\{\sigma(x_n^{m+1})\}$ where $x_n^{m+1} = x_n^m + \Delta t\, v_n^m$, which are used in the acceptance-rejection probabilities in each $\mathbb{E}_{p^{m+1}}$; see~\eqref{eq:probab}.  Note that $\mathcal{R}_n^m$ is conditioned on $(x_n^m, v_n^m)$, so it can be seen as a function of $(x_n^m, v_n^m)$ and consequently a function of  $\sigma(x_n^{m+1})$ for a given function $\sigma(x)$.  Thus, using the score function~\cite{Blei2017,mohamed2019monte,yang2022adjoint}, we can express the derivative of $\mathcal{R}_n^m$ with respect to each $\sigma(x_n^{m+1})$ as
\begin{equation}\label{eq:dRidsigma}
    \frac{\partial \mathcal{R}_n^m}{\partial \sigma(x_n^{m+1})} =  \tilde{\bbE}^{m} \left[ \frac{\partial \log \kappa(x_n^{m+1} ) }{\partial \sigma(x_n^{m+1}) }   r(x_n^M,v_n^M)  \right]  ,
\end{equation}
where the probability for the rejection sampling and its score function are
\[
\kappa(x_n^{m+1} )  = 
    \begin{cases}
   \alpha_n^{m+1} ,  & \text{if\ }v_n^{m+1} = v_n^m  \\
    1- \alpha_n^{m+1} , & \text{otherwise}
    \end{cases} \qquad    
 \frac{\partial \log \kappa(x_n^{m+1} ) }{\partial \sigma(x_n^{m+1}) }   = 
    \begin{cases}
   -\Delta t ,  & \text{if\ }v_n^{m+1} = v_n^m  \\
   \Delta t \frac{ \alpha_n^{m+1}}{1-\alpha_n^{m+1}} , & \text{otherwise}
    \end{cases} 
\]
Using the same trajectories from the MC solver for the forward RTE (see~\Cref{alg:f-RTE}), we obtain samples $\{ \widehat {\mathcal{R}}_n^m\}_{m=0}^{M-1}$ based on~\eqref{eq:Ri_def}, and the Monte Carlo gradient $\{  \widehat {\mathcal{G}}_n^m\}_{m=0}^{M-1}$ based on~\eqref{eq:dRidsigma}, respectively, where
\begin{equation}\label{eq:dRidsigma_discrete}
\widehat {\mathcal{R}}_n^m =  r(x_n^M, v_n^M),\qquad   \widehat {\mathcal{G}}_n^m  = 
     \begin{cases}
     -   r(x_n^M,v_n^M) \Delta t , & \text{if\ }  v_n^{m+1} = v_n^m , \\
      r(x_n^M,v_n^M) \dfrac{ \alpha_n^{m+1}  \Delta t }{1 - \alpha_n^{m+1} }, &\text{otherwise},
     \end{cases}
\end{equation}
for all $m = 0,\ldots,M-1$.

Based on~\eqref{eq:Ji_conditional} and~\eqref{eq:Ri_def}, we can also treat $\{\frac{1}{N} \widehat {\mathcal{R}}_n^m\}_{m=0}^{M-1}$ as samples of $J_n$. Given the same RTE particle trajectories obtained from~\Cref{alg:f-RTE}, $\{(x_n^m, v_n^m)\}_{m=0}^M$, the following holds for any $m = 0,\ldots, M-1$:
\begin{eqnarray*}
J &=& \sum_{n=1}^N J_n\quad  \stackrel{\text{sample~\eqref{eq:Ji_conditional}}}{\approx} \quad \frac{1}{N}\sum_{n=1}^N  {\mathcal{R}}_n^m \quad  \stackrel{\text{sample~\eqref{eq:Ri_def}}}\approx \quad   \frac{1}{N}\sum_{n=1}^N \widehat {\mathcal{R}}_n^m,\\
  &\stackrel{\text{based on~\eqref{eq:J2_MC}}}\approx& \mathcal J  = \frac{1}{N}\sum_{n=1}^N r(x_n^M, v_n^M).
\end{eqnarray*}
Thus, the Monte Carlo gradient  of  $\mathcal{J}$ with respect to $\sigma(x_n^{m+1})$ can be approximated by
\begin{equation}\label{eq:DTO_gradient_old}
\frac{\partial \mathcal J }{\partial \sigma(x_n^{m+1}) }\approx \frac{1}{N} \sum_{n=1}^N \frac{\partial  {\mathcal{R}}_{n}^m }{\partial \sigma(x_n^{m+1}) }  = \frac{1}{N} \frac{\partial \mathcal{R}_n^m }{\partial \sigma(x_n^{m+1}) } \approx  \frac{1}{N}\widehat{\mathcal{G}}_n^m,\quad \forall m,
\end{equation}
since $\mathcal{R}_{i}^{m}$ does not depend on $\sigma(x_n^{m+1})$ if $i\neq n$, $\forall m$. Combining \eqref{eq:DTO_gradient_old} with~\eqref{eq:dRidsigma_discrete}, we obtain a Monte Carlo gradient formula for $\frac{\partial \mathcal J }{\partial \sigma(x_n^{m+1})}$.

However, it is worth noting that 
\begin{equation}\label{eq:two gradients}
 \frac{\partial \mathcal{J}}{\partial \sigma(x_n^{m+1})}\neq \frac{\delta J}{\delta \sigma}(x_n^{m+1}).
\end{equation}
The left-hand side gradient treats $\sigma(x_n^{m+1})$ as a single parameter, and therefore the effective parameters are a collection of $MN$ number of scalars, $\{\sigma(x_n^{m+1})\}$, where $n=0,\ldots,N$ and $m = 0,\ldots M-1$. On the other hand, the right-hand side is a functional derivative with respect to the parameter function $\sigma(x)$ (the same with the OTD derivative~\eqref{eqn:frechet_derivative}) evaluated at $x= x_n^{m+1}$. We use the following example to show how to relate both sides of~\eqref{eq:two gradients}.

Consider a separate mesh grid $\{\bar{x}_j\}$ in the spatial domain, and $Q_j$ is a small neighborhood of $\bar{x}_j$ for each $j$. Consider a perturbed parameter function $\sigma(x) + \delta \sigma(x)$, where $\delta \sigma(x) = \mathds{1}_{x \in Q_j}  \epsilon$ for some small constant $\epsilon$. We then have
\begin{equation}\label{eq:perturb1}
J(\sigma  +\delta \sigma ) - J(\sigma) \approx \int_x \frac{\delta J}{\delta \sigma}(x) \delta \sigma(x) \rd{x} = \epsilon  \int_{Q_j} \frac{\delta J}{\delta \sigma}(x) \rd{x}.
\end{equation}
On the other hand, with $\mathcal{J}(\sigma)$ denoting the value~\eqref{eq:J2_MC} calculated with a given parameter function $\sigma$, we have the following based on~\eqref{eq:DTO_gradient_old}:
\begin{equation}\label{eq:perturb2}
\mathcal{J}(\sigma  +\delta \sigma ) -  \mathcal{J}(\sigma) \approx  \sum_{n=1}^N \sum_{m=0}^{M-1}  \frac{\partial \mathcal{J} }{\partial \sigma(x_n^{m+1})} \delta \sigma(x_n^{m+1})  \approx 
\epsilon\,  \frac{1}{N} \sum_{n=1}^N \sum_{m=0}^{M-1}  \mathds{1}_{x_n^{m+1} \in Q_j} \,  \widehat{\mathcal{G}}_n^m.
\end{equation}
Combining the last terms in~\eqref{eq:perturb1}-\eqref{eq:perturb2} and assuming the approximated gradient function $\frac{\delta J}{\delta \sigma}$ is piecewise constant on $\{Q_j\}$, we have 
\[
 \frac{\delta J}{\delta \sigma}(\bar x_j)\, |Q_j| \approx \int_{Q_j} \frac{\delta J}{\delta \sigma}(x) \rd{x} \approx \frac{1}{N} \sum_{n=1}^N \sum_{m=0}^{M-1}  \mathds{1}_{x_n^{m+1} \in Q_j} \, \widehat{\mathcal{G}}_i^n.
\]
% We may approximate 
% \begin{eqnarray}
% \frac{\partial J_i }{\partial \sigma(\bf x_j)}  \approx
% \frac{1}{|B_j|}  \sum_{k=1}^{M} \int_{B_j} \frac{\partial J_i}{\partial \sigma(x_n^{k})}  d\mu(x_n^{k})  
% % &=&  \frac{1}{|B_j|} \sum_{k=0}^{M-1}  \int  \int_{B_j} \mathds{1}_{x_n^k \in B_j} \left(  \frac{1}{N} \mathbb{E}_i^{1} \ldots \mathbb{E}_i^{M}  \left[ \frac{\partial \log p(x_n^{k+1} ) }{\partial \sigma(x_n^{k+1}) }   r(x_n^M,v_n^M) \Big|  (v_n^{0},x_n^{0} ) \right]  \right)f(x^0_i, v^0_i, 0) dx^0_i dv^0_i \nonumber \\
% \approx   \frac{1}{|B_j|} \sum_{k=1}^{M} \mathds{1}_{x_n^{k} \in B_j} \frac{\partial J_i }{\partial \sigma(x_n^{k})},
% \end{eqnarray}
% where the last term is obtained after sampling. 
Finally, we may approximate the gradient using values of~\eqref{eq:dRidsigma_discrete} by
\begin{eqnarray}
 \frac{\delta J}{\delta \sigma}(\bar x_j)  & \approx& \frac{1}{|Q_j|}  \frac{1}{N}  \sum_{n=1}^N \sum_{m=0}^{M-1}  \mathds{1}_{x_n^{m+1} \in Q_j}   \widehat{\mathcal{G}}_n^m\nonumber
 \\
 &=&\frac{1}{|Q_j|} \frac{ \Delta t}{N}\sum_{n=1}^N\sum_{m=1}^{M} \ \mathds{1}_{x_n^{m} \in Q_j} \  r(x_n^M, v_n^M) \ \xi_n^m,\quad \text{where\ } \xi_n^m = \begin{cases}
 -1 , & \text{if\ }  v_n^{m} = v_n^{m-1}, \\
 \frac{\alpha_n^m}{1 - \alpha_n^m } , &\text{otherwise}.
 \end{cases} \label{eq:DTO_grad}
\end{eqnarray}
Similar to the particle-based OTD approach presented in~\Cref{alg:P-OTD}, sorting is also needed in computing \eqref{eq:DTO_grad} for all particles in the long time horizon, and leads to $\mathcal{O}(NM (\log N + \log M))$ complexity. Note that this complexity can be reduced to $\mathcal{O}(NM (\log N))$ if one sorts the particle at each time step instead of doing it all at once in the end, which reduces it to the same complexity as the OTD approach.

We summarize steps of this particle-based  DTO approach for gradient calculation in~\Cref{alg:P-DTO}. 
\begin{algorithm}
\caption{The Particle-Based DTO Approach for Gradient Computation\label{alg:P-DTO}}
\begin{algorithmic}[1]
\State  Given cell centers $\{\bar{x}_j\}$ where we want to evaluate the gradient~\eqref{eqn:frechet_derivative}.
\State Implement~\Cref{alg:f-RTE} to solve the forward RTE~\eqref{eqn:RTE} and store the trajectories $\{(x_n^m,v_n^m)\}$ in memory where $m=0,\ldots,M$ and $n=1,\ldots,N$.  
\State Compute the objective function~\eqref{eq:J2_MC} based on $\numf^M$ and evaluate $\widehat{\mathcal{G}}_n^m$ following~\eqref{eq:dRidsigma_discrete} for all $m$ and $n$.
\State Evaluate the gradient at $x=\bar{x}_j$ following~\eqref{eq:DTO_grad}.
\end{algorithmic}
\end{algorithm}

\section{Numerical Examples}\label{sec:tests}
In this section, we present a few numerical tests to illustrate gradient computed by particle methods following the OTD and DTO approaches presented in~\Cref{sec:OTD} and~\Cref{sec:DTO}. We will refer to them as P-OTD and P-DTO in this section. As a reference, we will use a forward Euler scheme along with an upwind spatial discretization for both forward and adjoint equations. Details are provided in~\Cref{sec:FVM}. 
We remark that the gradient calculation based on the finite-volume method (FVM) also belongs to the OTD approach for which we use a consistent finite-volume upwind scheme (adjoint with respect to the FVM for solving~\eqref{eqn:RTE}) to discretize the adjoint equation~\eqref{eqn:g_eqn}. Thus, it also coincides with what one would get from the DTO approach, with the discretization being the FVM. 

\subsection{Inverse Problem}\label{sec:test_ip}
First, we consider a setup based on inverse data matching problems similar to the one described in~\Cref{subsec:setup}. We measure the spatial-domain density at the final time $T$, and the reference probability density function is $d(x)$. For simplicity, we choose the $L^2$-based objective functional
\begin{equation}\label{eq:J1}
    J_1(\sigma) = \frac{1}{2}\int_D |\rho_T(x) - d(x)|^2 \rd{x},\quad \rho_T(x) = \frac{1}{|\Omega|} \int_\Omega f(T,x,v)\rd{v},
\end{equation}
where other proper metrics and divergence for the probability space could also be considered. 

It is worth noting that~\eqref{eq:J1} is different from~\eqref{eqn:min}. Recall in~\Cref{sec:OTD} where we use the method of Lagrange multipliers to derive the equation for $g$ shown in~\eqref{eqn:g_eqn}. The final condition for $g$ will change with respect to the functional $J$ evaluated at the final-time RTE solution $f(T,x,v)$, while the back-propagation rule for $g$ (i.e., the PDE itself) is independent of the choice of $J$. Hence, the final-time condition of $g$ for the objective function $J_1$ should be
\[
g(T,x,v)  = -\frac{\delta J_1}{\delta f(T,x,v)} =|\Omega|^{-1} \left(d(x) - \rho_T(x)\right),
\]
which is constant in $v$.

\begin{figure}
\centering
\subfloat[P-OTD and FVM gradients]{    \includegraphics[width=.49\linewidth]{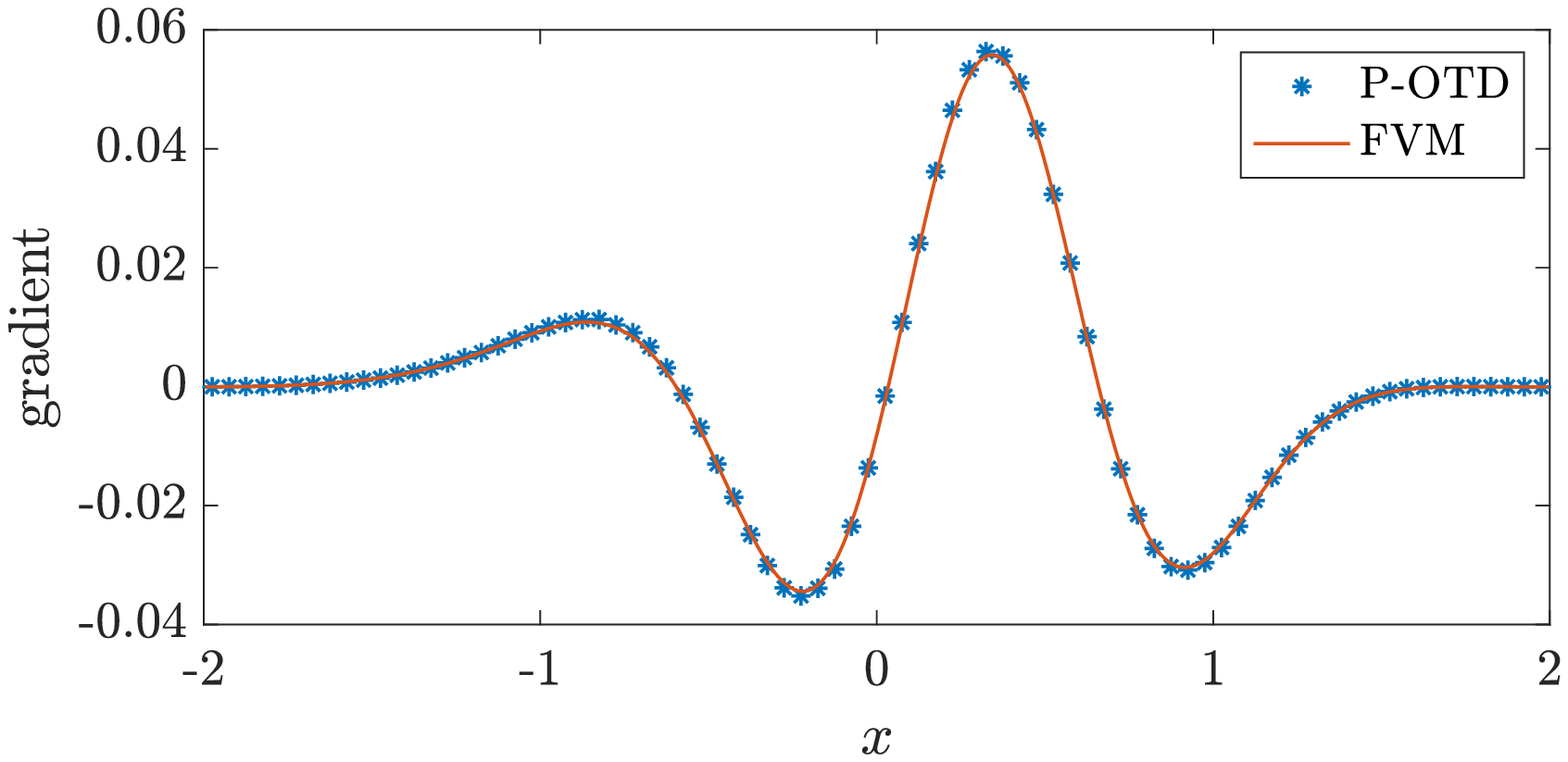}\label{fig:test-1D-1-comp}}
\subfloat[2-norm error between the two gradients]{    \includegraphics[width=.49\linewidth]{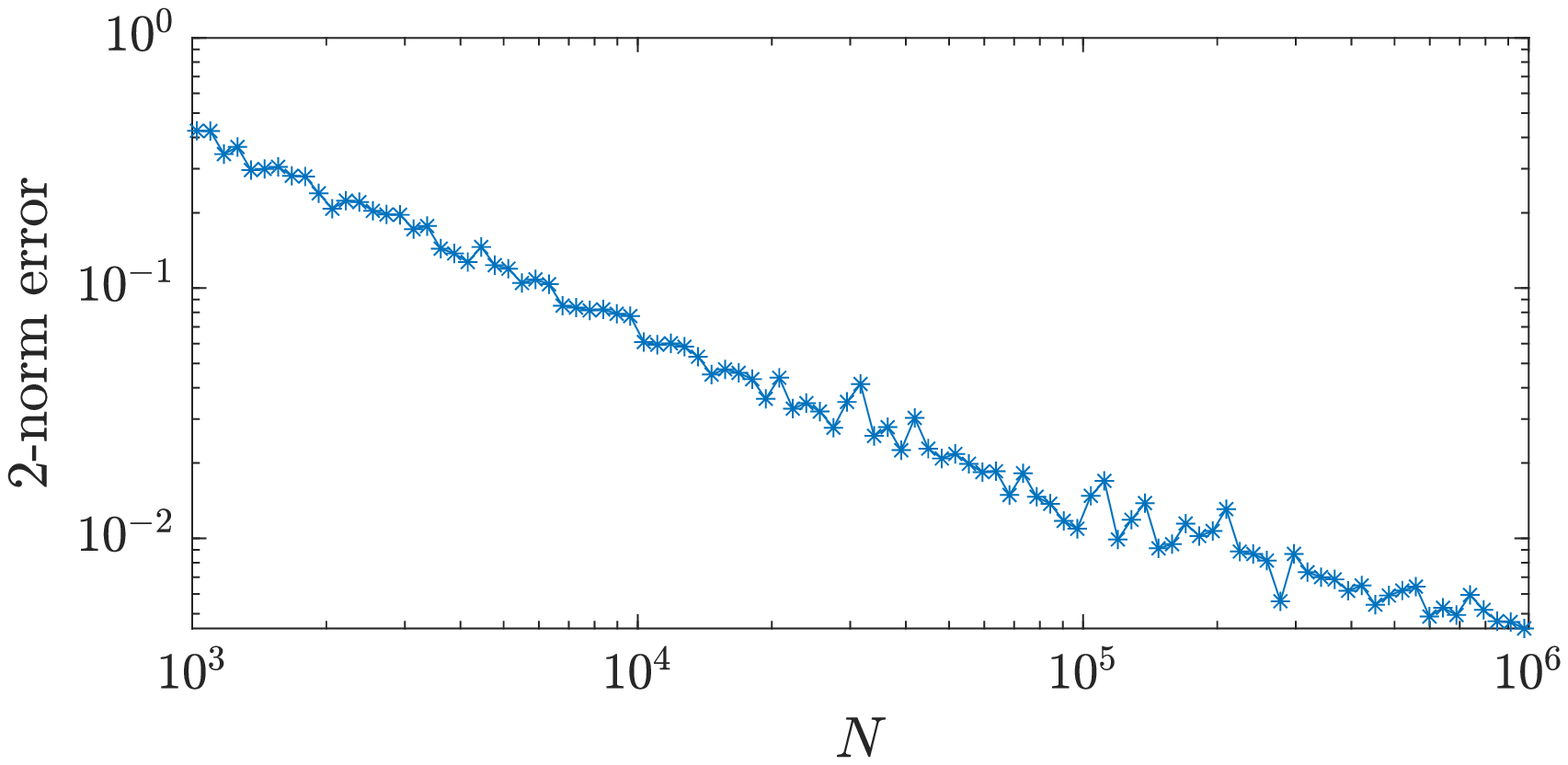}\label{fig:test-1D-1-std}}
\caption{Comparison between the gradients calculated by FVM and P-OTD for the objective function~\eqref{eq:J1} in a 1D setting. (a): Illustration of the two gradients using $N=10^6$ particles in P-OTD. (b): The 2-norm error between the two gradients with respect to the number of particles $N$ in P-OTD.}%\label{fig:test-1D-1}
\end{figure}

In a 1D setting, we consider the spatial domain $D = [-2,2]$ and the velocity domain $\Omega = [-1,1]$, with the periodic boundary condition on $D$. We run the forward RTE~\eqref{eqn:RTE} on the time interval $[0,T]$ where $T=2$. The step size $\Delta t = 0.01$ in both FVM and P-OTD. The initial distribution is 
\[
f_\In(x,v) = 2 \pi^{-\frac{1}{2}} \exp(-4 x^2)\,,
\]
which is constant in $v$, and the measurement density is
\[
d(x) =  \sqrt{5}\pi^{-\frac{1}{2}} \exp(-5|x-0.6|^2)\,.
\]
Note that $\iint f_\In(x,v) dx dv = 2$ and $\int d(x) dx = 1$. We evaluate the gradient at $\sigma(x) = 2$ for any $x\in D$. In~\Cref{fig:test-1D-1-comp}, we illustrate the gradients computed by FVM and P-OTD (using $N=10^6$ number of particles), respectively, while in~\Cref{fig:test-1D-1-std}, we show in a log-log plot the $2$-norm error between gradients computed by the two methods as the number of particles $N$ used in P-OTD increases. The error decay demonstrates the expected Monte Carlo error of $\mathcal{O}(1/\sqrt{N})$.

In a 2D setting, we consider spatial domain $D = [-1, 1]^2$ and velocity domain $\Omega = \mathbb{S}^1$, again with the periodic boundary condition on $D$. We parameterize the velocity using the polar coordinate, $v = [\cos \theta, \sin \theta]^\top$, $\theta \in [-\pi, \pi]$. The initial distribution 
\[
f_\In(x,v) = 4 \pi^{-1}\exp(-4|x|^2)\,,
\]
which is again constant in $v$, and the measurement density is
\[
d(x) =  5\pi^{-1} \exp(-5|x_1-0.3|^2-5|x_2+0.3|^2)\,,\quad x = [x_1,x_2]^\top\,.
\]
Note that $\iint f_\In(x,v) dx dv = 2\pi$ and $\int d(x) dx = 1$. We evaluate the gradient at $\sigma(x) = 2,\,\forall x\in D$. The time step $\Delta t = 0.01$ and the final time $T=0.5$. The two gradients are shown in~\Cref{fig:test-2D-1}. We used $N=10^6$ particles in the P-OTD method and plotted the averaged value from $100$ i.i.d.~runs in~\Cref{fig:test-2D-1-P-OTD} to further reduce the random error by a factor of $10$.

\begin{figure}
\centering
\subfloat[FVM gradient]{\includegraphics[width=.4\linewidth]{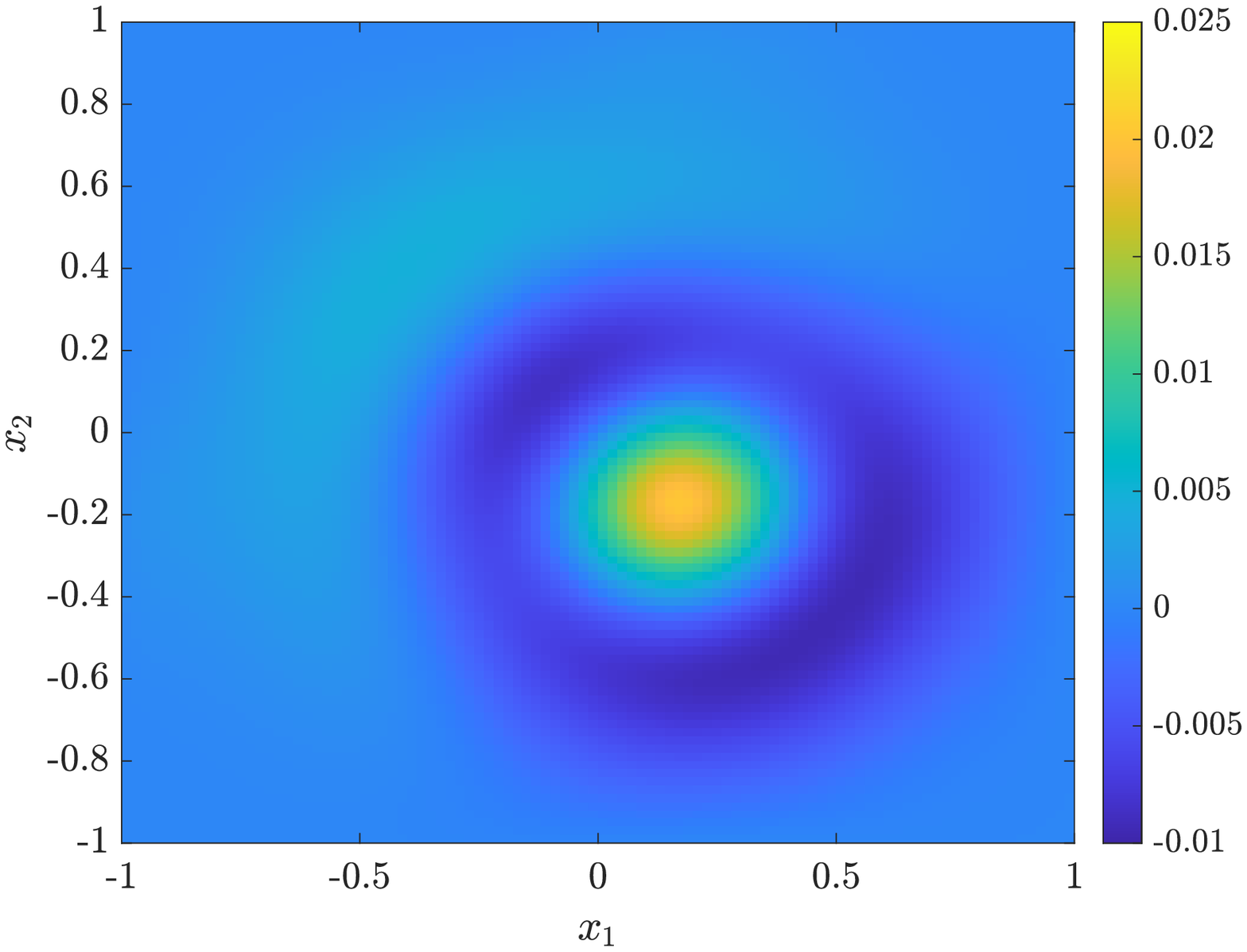}\label{fig:test-2D-1-FVM}}
\hspace{0.08\textwidth}
\subfloat[P-OTD gradient]{\includegraphics[width=.4\linewidth]{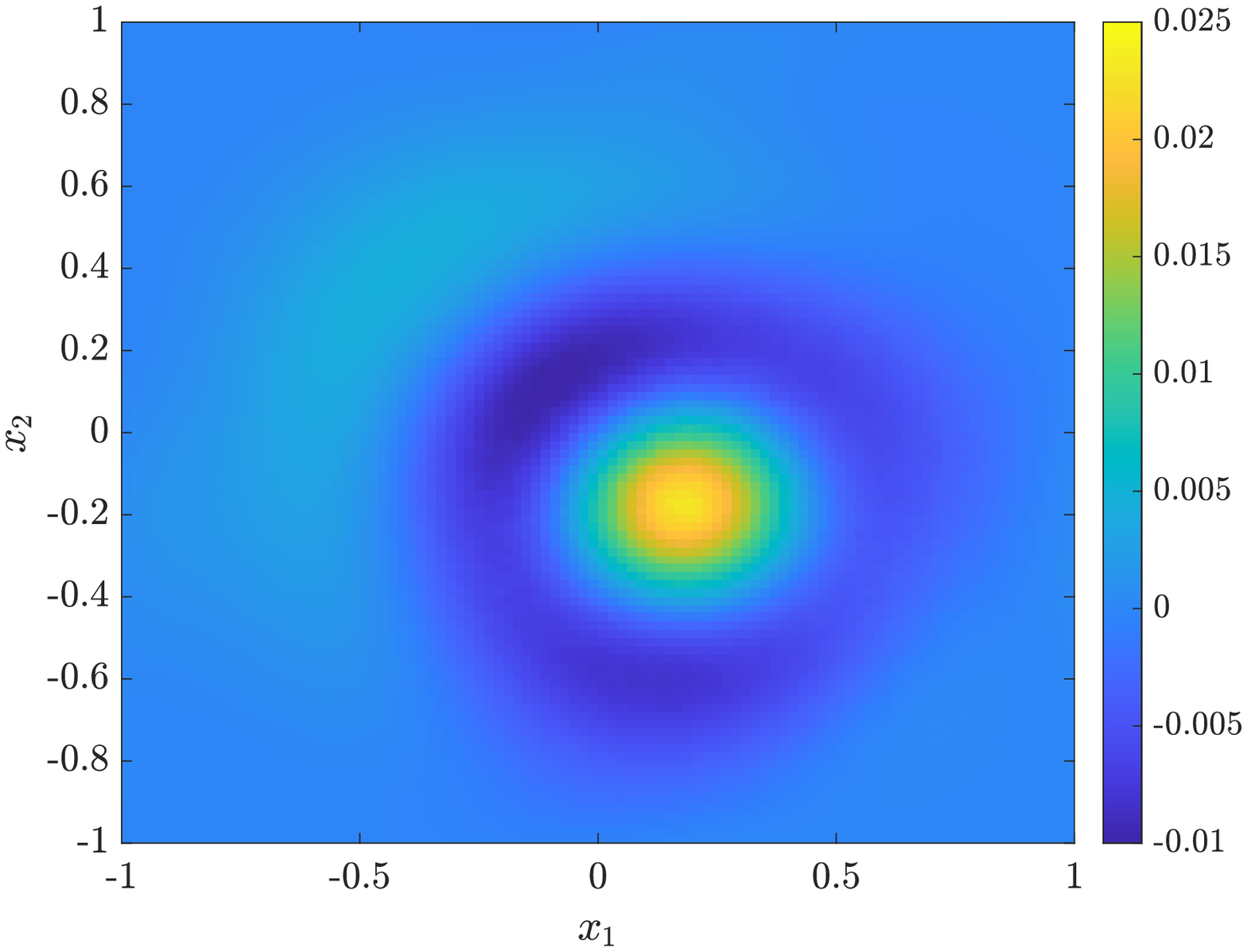}\label{fig:test-2D-1-P-OTD}}
\caption{Comparison between the gradients calculated by FVM (a) and P-OTD (b) for the objective function~\eqref{eq:J1} in a 2D setting described in~\Cref{sec:test_ip}. The number of particles in P-OTD is $N=10^6$.}
    \label{fig:test-2D-1}
\end{figure}

\subsection{Optimal Control}
In this subsection, we focus on a different objective functional~\eqref{eq:J2} which we will refer to as $J_2$.
% \begin{equation}\label{eq:J2}
%   J_2(\sigma)= \iint r(x,v) f(T,x,v) dx dv,
% \end{equation}
It is often used in optimal control or optimal design applications. Here, we measure the macroscopic quantity at the final time $T$ for function 
\[
r(x,v) = s(v) I_{E}(x),
\]
where $I_{E}(x)$ is an smooth approximation to the indicator function $\mathds{1}_E(x) = \mathds{1}_{x\in E}$ for a chosen measurement set $E\subset D$; see~\Cref{fig:IE} for illustrations of $I_{E}(x)$ in 1D and 2D settings. Based on $J_2$, the final condition of the adjoint variable $g$ should be 
$$g(T,x,v) = -\frac{\delta J_2}{\delta f(T,x,v)} = -r(x,v).
$$
The initial conditions for the RTE remain the same as the examples in~\Cref{sec:test_ip}. 

We use three methods to compute the gradient of $J_2$ at $\sigma(x) = 2$: FVM, P-OTD and P-DTO. The first two methods follow the OTD approach since they discretize the forward RTE~\eqref{eqn:RTE} and the continuous adjoint equation~\eqref{eqn:g_eqn}, whose solutions are plugged into~\eqref{eq:OTD_grad} for gradient calculation. On the other hand, the P-DTO method derived in~\Cref{sec:DTO} follows the DTO approach and it does not solve the adjoint equation~\eqref{eqn:g_eqn}. Using the solution based on~\Cref{alg:f-RTE} to the forward RTE~\eqref{eqn:RTE} and the history of rejection samplings therein, the gradient can be approximated by formula~\eqref{eq:DTO_grad}.

In 1D, we set $s(v) = v^2$, $T=0.5$ and $\Delta t = 0.005$. \Cref{fig:test-1D-2} illustrates the comparison among the three methods. In the 2D case, we consider the spatial domain $D = [-1.5, 1.5]^2$ and the velocity domain $\Omega = \mathbb{S}^1$, and set $s(v) = |v_1|^2$ where $v = [v_1,v_2]^\top$. The total simulation time $T = 0.2$ while $\Delta t = 0.005$. We show the gradients calculated from the three methods in~\Cref{fig:test-2D-2}. We further reduce the variances in the P-OTD gradient by taking its averaged value after $100$ i.i.d.~runs. We also average the P-DTO gradient based on $300$ i.i.d.~runs.

\begin{figure}
\centering
\subfloat[1D setting]{    \includegraphics[width=.45\linewidth]{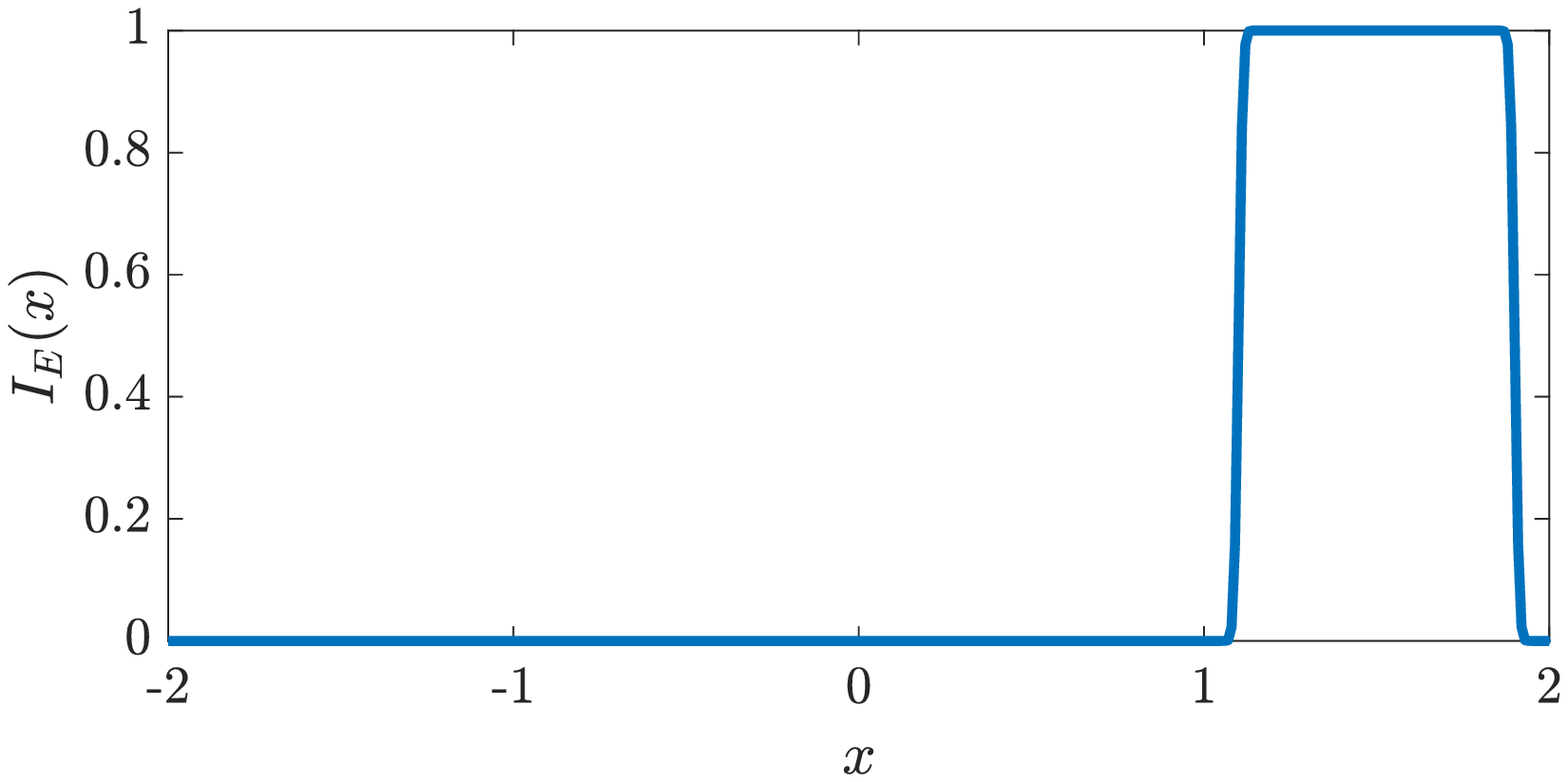}\label{fig:IE-1d}}
\hspace{0.06\linewidth}
\subfloat[2D setting]{    \includegraphics[width=.45\linewidth]{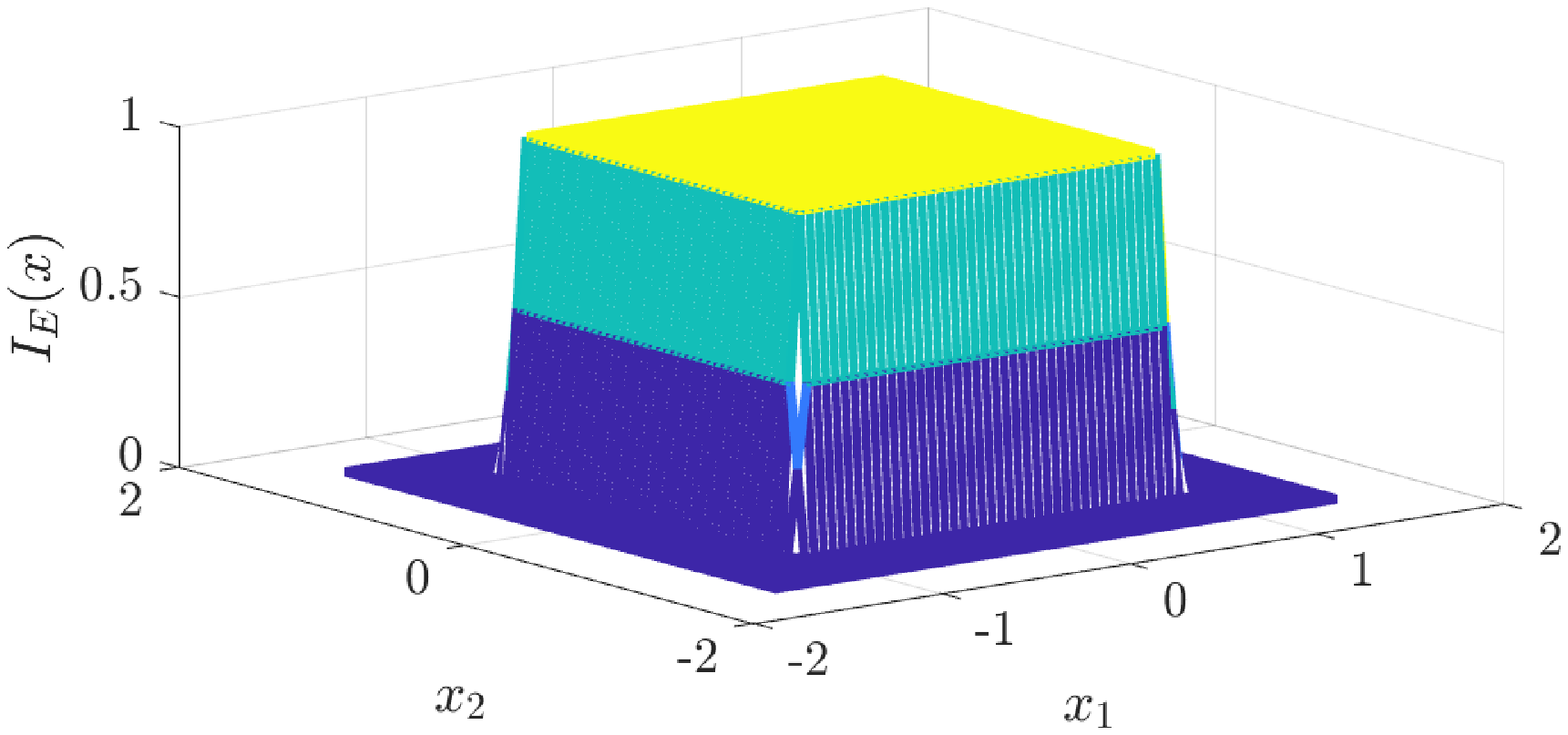}\label{fig:IE-2d}}
\caption{Illustration of approximated indicator function $I_E(x)$ in 1D (a) and 2D (b) settings.\label{fig:IE}}
\end{figure}

\begin{figure}
\centering
\subfloat[FVM vs. P-OTD]{    \includegraphics[width=.45\linewidth]{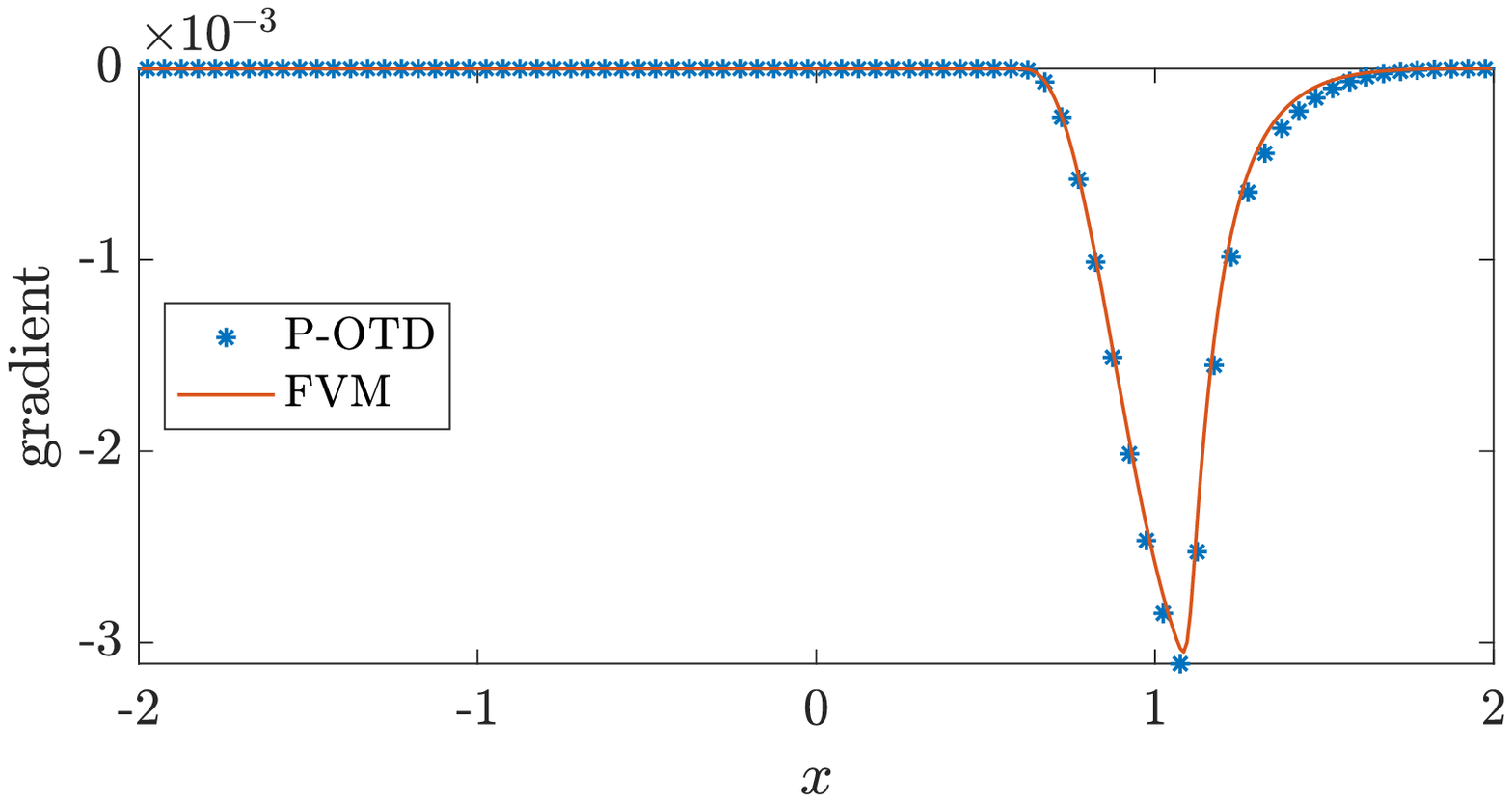}\label{fig:test-1D-2-1}}
\hspace{0.06\linewidth}
\subfloat[FVM vs. P-DTO]{    \includegraphics[width=.45\linewidth]{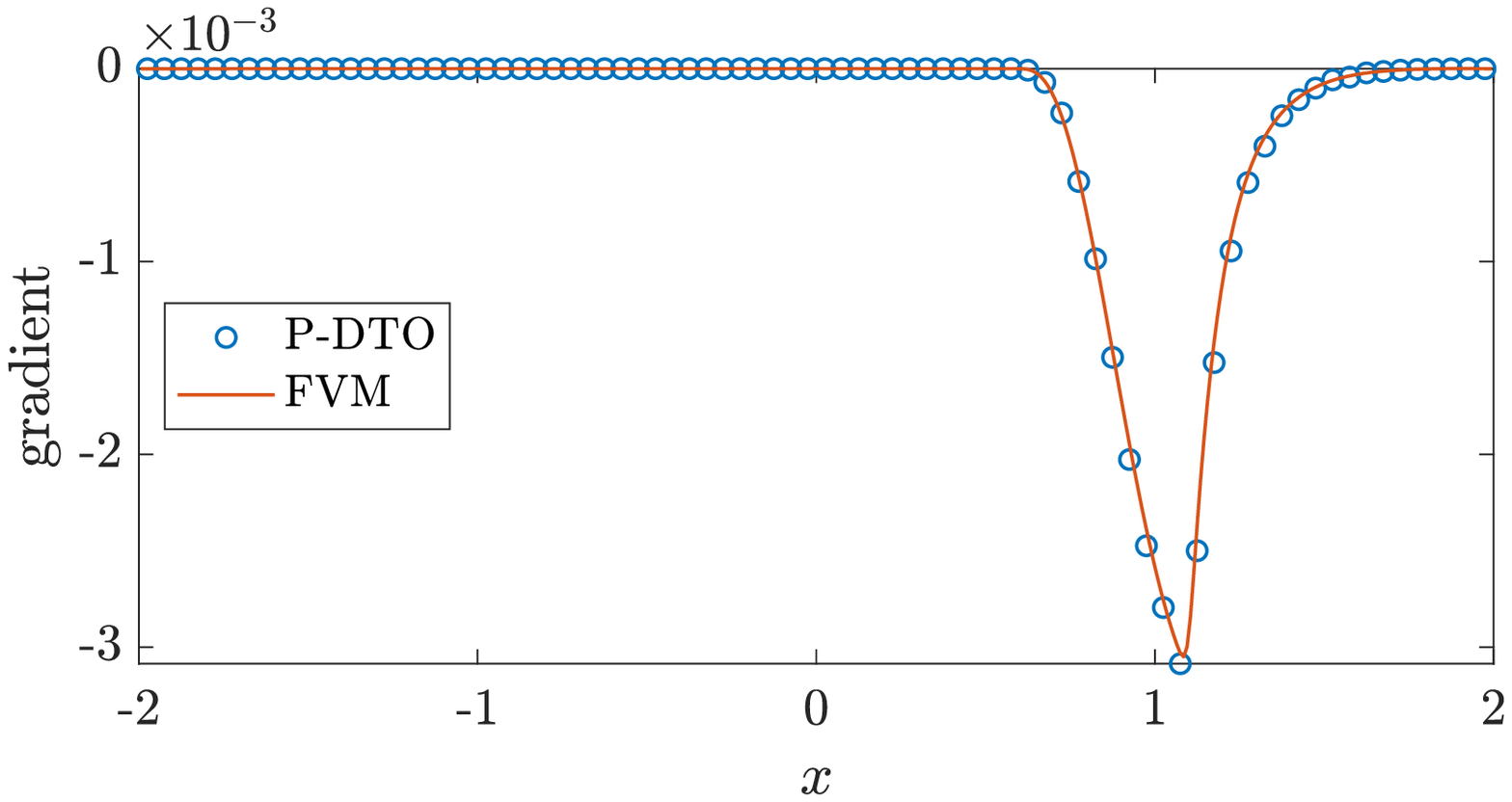}\label{fig:test-1D-2-2}}
\caption{Gradient comparisons among FVM, P-OTD and P-DTO methods in a 1D setting for the objective function $J_2$. The number of particles in both particle methods is $N=10^6$.\label{fig:test-1D-2}}
\end{figure}

\begin{figure}
\centering
\subfloat[P-OTD]{    \includegraphics[width=.32\linewidth]{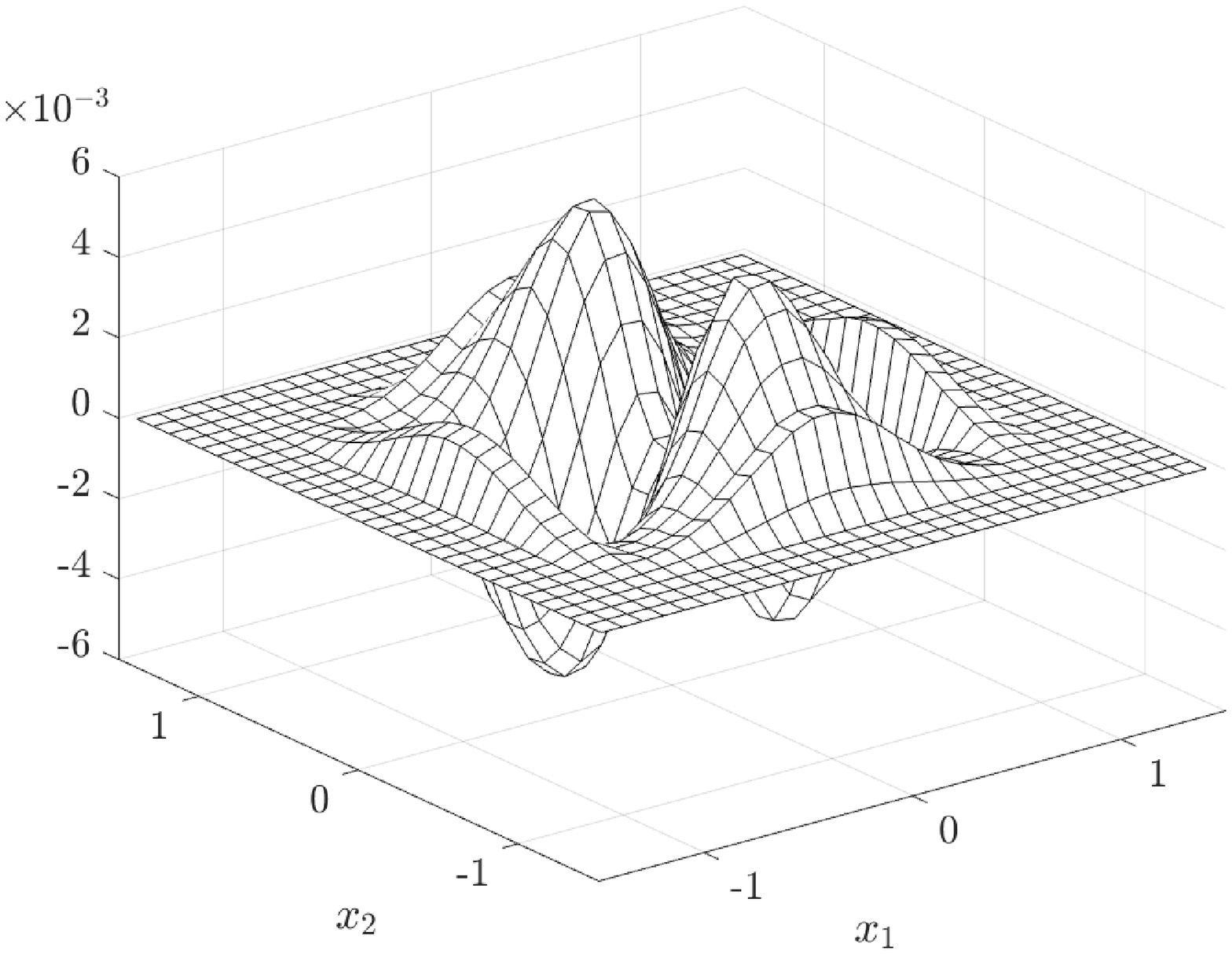}\label{fig:test-2D-2-P-OTD}}
\subfloat[P-OTD]{    \includegraphics[width=.32\linewidth]{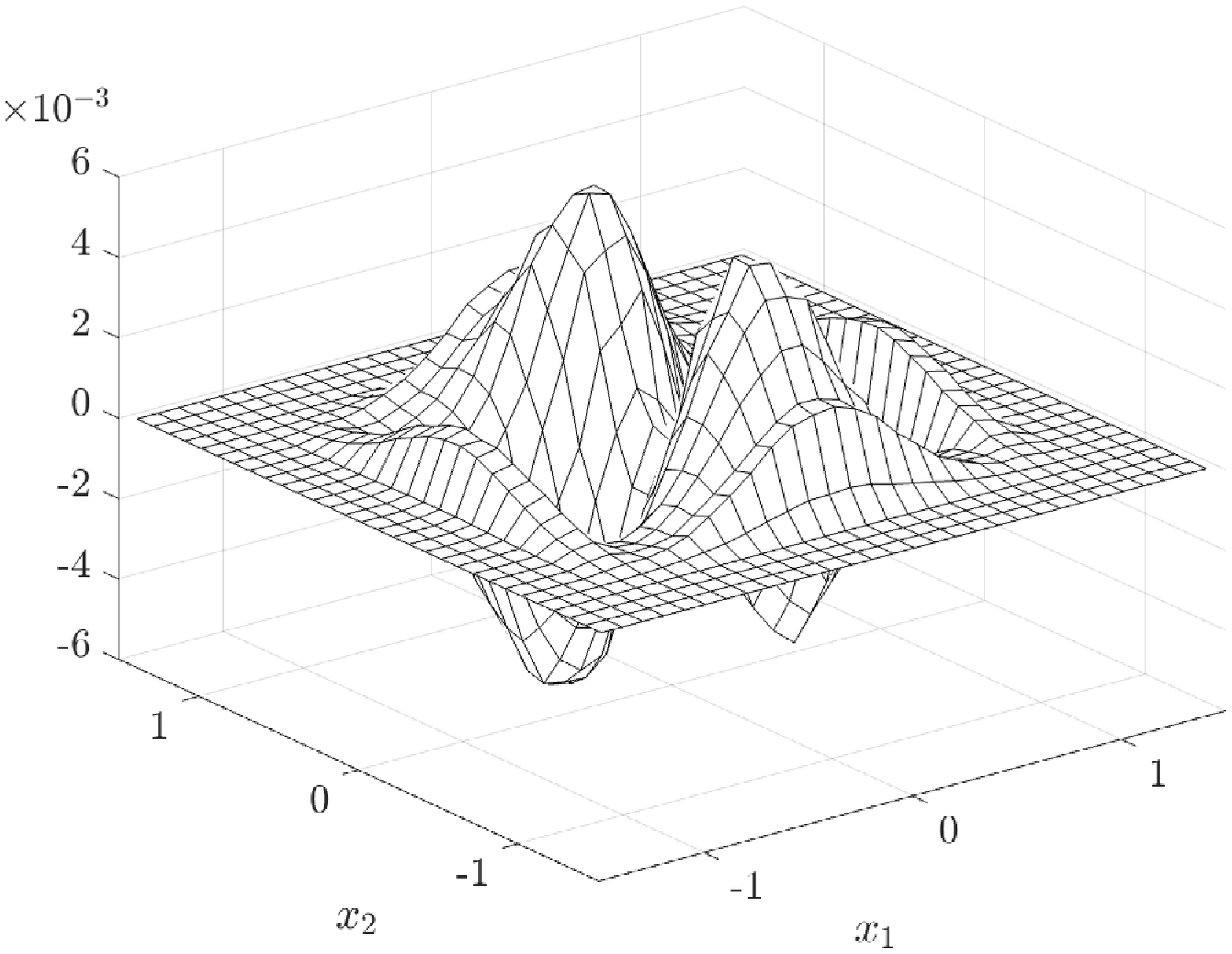}\label{fig:test-2D-2-P-DTO}}
\subfloat[FVM]{\includegraphics[width=.32\linewidth]{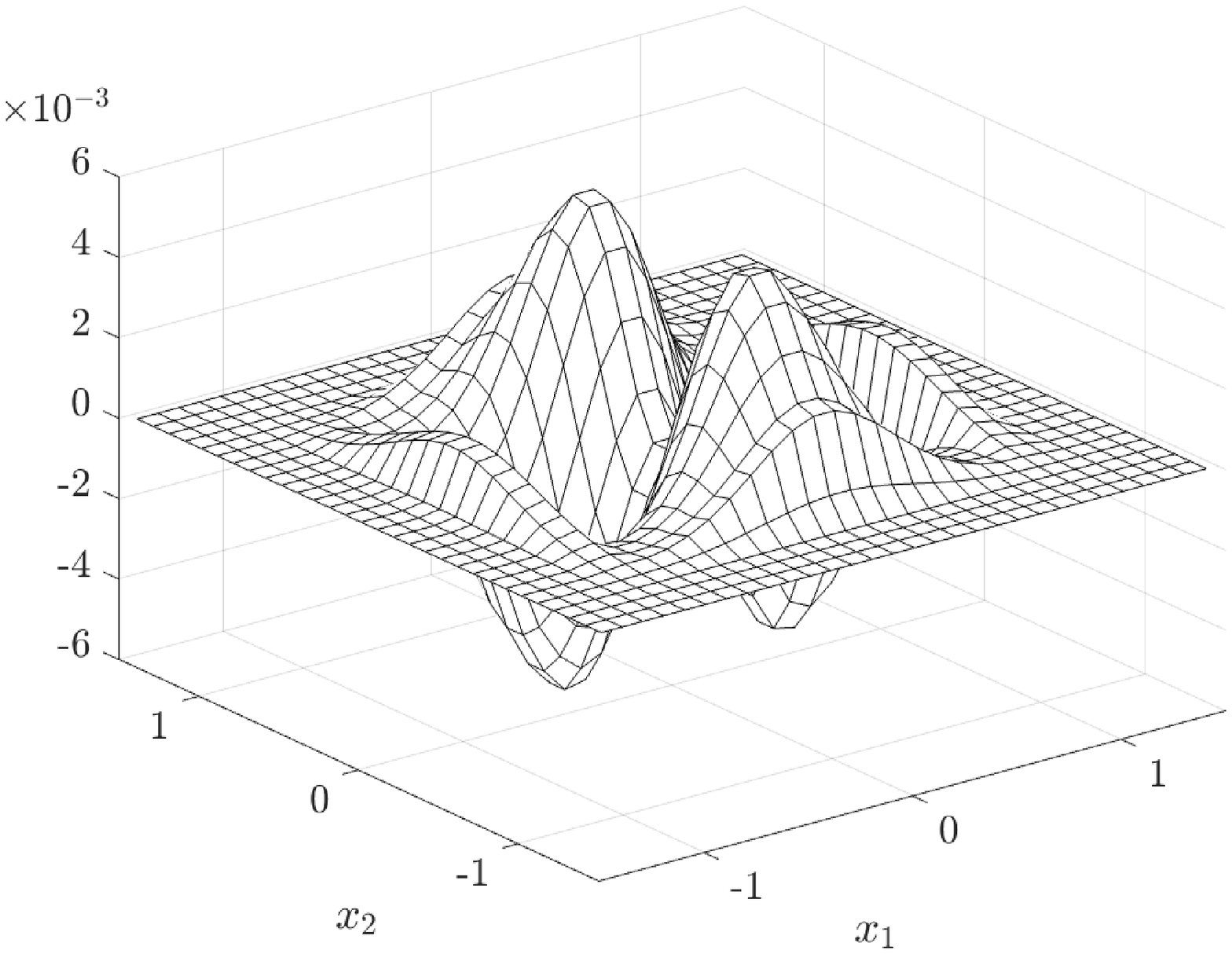}\label{fig:test-2D-2-FVM}}
\caption{Gradient comparisons among FVM, P-OTD and P-DTO methods in a 2D setting for the objective function $J_2$. The number of particles in both particle methods is $N=10^6$.\label{fig:test-2D-2}}
\end{figure}

%\section{Conclusions} \label{sec:conclusions}

\bibliography{references}
\bibliographystyle{siamplain}

\appendix
\section{The Finite Volume Scheme}\label{sec:FVM}
We summarize the finite volume scheme used to compute the reference solution. As an illustration, we only consider the 1D case. Denote
\[
\fij^m \approx f(t^m, x_i, v_j)\,, \qquad 0\leq m \leq M, \quad 1\leq j \leq N_v\,,
\]
as the numerical approximation, and let  $\dx$, $\dt$ and  $\dv $ be the corresponding mesh size in $x$, $t$ and $v$, respectively. We consider the time domain $[0,T]$, the spatial domain $D = [-2,2]$ and the velocity domain $\Omega = [-1,1]$. Then $M \dt = T$, and $\dv N_v = |\Omega| = 2$, where $T$ is the final time. We have the following discretization for \eqref{eqn:RTE}:
\begin{align*}
    \fij^\np - \fij^m + \frac{\dt}{\dx} v_j^{+}(\fij^m - \fijm^m) 
    + \frac{\dt}{\dx} v_j^{-}(\fijp^m - \fij^m) = \sigma_i \dt \left(|\Omega|^{-1}\average{\fij^m} - \fij^m \right)
    \,,
\end{align*}
for $0\leq m \leq M -1$ and $1\leq j \leq N_v$, 
with initial condition 
\begin{align*}
    \fij^0 = f_\In(x_i, v_j)\,.
\end{align*}
Here, $v_j = -1+ (j-1/2) \dv$ representing the cell center, and $v^+ = \max\{v,0\}$, $v^- = \min\{v,0\}$. The average in $v$ is computed by a simple midpoint rule: $\average{\fij^n} = \sum_{j=1}^{N_v} \fij^n \dv$. 

Writing down the discrete version of the objective function (taking \eqref{eq:J1} as an example)
\begin{align*}
J = & \frac{\dx }{2} \sum_i \left(d_i - |\Omega|^{-1} \average{f^M_{i,j}}\right)^2 + \dx  \dv \sum_{i} \sum_{m=1}^{M-1} \sum_{j=1}^{N_v}  g_{i,j}^{m+1}\\
& \left(  \fij^\np - \fij^m + \frac{\dt}{\dx} v_j^{+}(\fij^m - \fijm^m) 
+ 
 \frac{\dt}{\dx} v_j^{-}(\fijp^m - \fij^m) - \sigma_i \dt \left(\frac{\average{\fij^m}}{|\Omega|} - \fij^m \right) \right), 
\end{align*}
then the optimality condition leads to the following discretization for \eqref{eqn:g_eqn}:
\begin{align*}
    \gij^m - \gij^\np + \frac{\dt}{\dx} v_j^{+}(\gij^m - \gijp^m) 
    + \frac{\dt}{\dx} v_j^{-}(\gijm^m - \gij^m) = \sigma_i \dt (|\Omega|^{-1}\average{\gij^\np} - \gij^\np)
    \,,
\end{align*}
for $M-1\geq m \geq 0, 1\leq j \leq N_v$, with final condition depending on the objective function:
\begin{align*}
    \gij^M = |\Omega|^{-1} \left( d_i - |\Omega|^{-1} \average{\fij^M}\right) \,,
\end{align*}
where $|\Omega|$ is the Lebesgue measure of the velocity domain. We use periodic boundary condition in $x$ throughout the whole calculation. Then it is straightforward to calculate the gradient  as 
% \begin{align*}
%     \frac{dJ}{d\sigma_i} = \sum_{n=1}^{M-1} \sum_{j=1}^{N_v} \fij^n \left(  \gij^\np - |\Omega|^{-1} \average{\gij^\np}\right)  \dv \dt\,.
% \end{align*}
\begin{align*}
    \frac{\delta J}{\delta \sigma_i}  = \dx \dv \dt \sum_{m=1}^{M-1} \sum_{j=1}^{N_v} \fij^m \left( \gij^\np - |\Omega|^{-1} \average{\gij^\np} \right) \,.
\end{align*}
Similar to~\eqref{eq:two gradients}, $\frac{\delta J}{\delta \sigma_i}$ above is not the same as $ \frac{\delta J}{\delta \sigma} (x_i)$, but we can relate these two following similar procedures in~\eqref{eq:perturb1}-\eqref{eq:perturb2}. Assuming that the gradient $\frac{\delta J}{\delta \sigma} (x)$ is piecewise constant over the spatial cells $\{[x_i - \dx/2,x_i + \dx/2 )\}$, we have that
\[
\frac{\delta J}{\delta \sigma} (x_i) \approx \frac{1}{\dx}   \frac{\delta J}{\delta \sigma_i}  = \dv \dt \sum_{m=1}^{M-1} \sum_{j=1}^{N_v} \fij^m \left( \gij^\np - |\Omega|^{-1} \average{\gij^\np} \right) \,.
\]

\end{document}